\documentclass[12pt,a4paper]{article}
\usepackage[cp1251]{inputenc}
\usepackage{amsmath,amsthm}
\usepackage{amsfonts,amstext,amssymb,verbatim,epsfig}
\usepackage{pgf,tikz}
\usetikzlibrary{arrows}
\usepackage{hyperref}
\def\R{{\mathbb{R}}}
\def\N{{\mathbb{N}}}
\def\Z{{\mathbb{Z}}}

\newcommand{\E}{\mathbb{E}}
\renewcommand{\P}{\mathbb{P}}
\def\e{{\varepsilon}}

\def\<{\langle}
\def\>{\rangle}

\newcommand{\Pro}{\ensuremath{\mathbb{P}}}

\newcommand{\Var}{\mathop{\mathrm{Var}}\nolimits}

\theoremstyle{usual}
\newtheorem{theorem}{Theorem}[section]
\newtheorem{corollary}{Corollary}[section]
\newtheorem{lemma}{Lemma}
\newtheorem{proposition}{Proposition}

\newtheoremstyle{likedef}
  {}%
  {}%
  {}%
  {\parindent}%
  {\bfseries}%
  {.}%
  {.5em}%
  {}%

\theoremstyle{likedef}
\newtheorem{definition}{Definition}[section]
\newtheorem{remark}{Remark}

\numberwithin{equation}{section}

\begin{document}

\title{Differentiability at the edge of the percolation cone and related results in first-passage percolation}

\author{Antonio Auffinger\thanks{Department of Mathematics, University of Chicago, 5734 S. University Avenue, Chicago, Illinois 60637. Email:auffing@math.uchicago.edu. Research partially funded by NSF Grant DMS 0806180.}
\and
Michael Damron\thanks{Mathematics Department, Princeton University, Fine Hall, Washington Rd., Princeton, NJ 08544. Email: mdamron@math.princeton.edu;
Research funded by an NSF Postdoctoral Fellowship.}}
\date{February 2012}
\maketitle


\begin{abstract}
We study first-passage percolation in two dimensions, using measures $\mu$ on passage times with $b:=\inf~supp(\mu) >0$ and $\mu(\{b\})=p\geq \vec p_c$, the threshold for oriented percolation. We first show that for each such $\mu$, the boundary of the limit shape for $\mu$ is differentiable at the endpoints of flat edges in the so-called percolation cone. We then conclude that the limit shape must be non-polygonal for all of these measures. Furthermore, the associated Richardson-type growth model admits infinite coexistence and if $\mu$ is not purely atomic the graph of infection has infinitely many ends. We go on to show that lower bounds for fluctuations of the passage time given by Newman-Piza extend to these measures. We establish a lower bound for the variance of the passage time to distance $n$ of order $\log n$ in any direction outside the percolation cone under a condition of finite exponential moments for $\mu$. This result confirms a prediction of Newman and Piza \cite{NewmanPiza} and Zhang \cite{Zhang}. Under the assumption of finite radius of curvature for the limit shape in these directions, we obtain a power-law lower bound for the variance and an inequality between the exponents $\chi$ and $\xi$.
\end{abstract}

\section{Introduction}
In this paper, we consider i.i.d. first-passage percolation (FPP) on $\mathbb{Z}^2$. This model was originally introduced by Hammersley and Welsh \cite{HW} in 1965 as a model of fluid flow through a random medium. It has been a topic of research for probabilists since its origins but despite all efforts through the past decades (see \cite{Blair-Stahn, Howard, Kestensurvey} for reviews) most of the predictions about its important statistics remain to be understood. 

The model is defined as follows. We place a non-negative random variable $\tau_e$, called the {\it passage time} of the edge $e$, at each nearest-neighbor edge in $\Z^2$. The collection $(\tau_e)$ is assumed to be i.i.d. with common distribution $\mu$.  A {\it path} $\gamma$ is a sequence of edges $e_1, e_2, \ldots$ in $\Z^2$ such that for each $n \geq 1$, $e_n$ and $e_{n+1}$ share exactly one endpoint. For any finite path $\gamma$ we define the {\it passage time} of $\gamma$ to be
\[
\tau(\gamma)=\sum_{e \in \gamma} \tau_e\ .
\]
Given two points $x,y \in \mathbb{R}^2$ one then sets
\[
\tau(x,y) = \inf_{\gamma} \tau(\gamma)\ ,
\]
where the infimum is over all finite paths $\gamma$ that contain both $x'$ and $y'$, and $x'$ is the unique vertex in $\mathbb{Z}^2$ such that $x \in x' + [0,1)^2$ (similarly for $y'$). A minimizing path for $\tau(x,y)$ is called a {\it geodesic} from $x$ to $y$. For each $t \geq 0$ let 
\[
B(t) = \{y \in \mathbb{R}^2~:~\tau(0,y)\leq t\}\ .
\]

In the case that $\mu(\{0\})=0$, the pair $(\mathbb{Z}^2,\tau(\cdot, \cdot))$ is a metric space and $B(t)$ is the (random) ball of radius $t$ around the origin. It is natural to ask about the large scale structure of this space, in particular how $B(t)$ typically looks as $t$ gets large. The classical and foundational ``shape theorem'' of Cox and Durrett \cite{CoxDurrett} gives the analogue of the law of large numbers for the growth of this random ball. It is the first step to understand this random metric. It can be stated as follows.

Let $\mathcal{M}$ be the set of Borel probability measures on $[0,\infty)$ with finite mean and with $\mu(\{0\})<p_c$, the threshold for bond percolation in $\Z^2$. If $S$ is a subset of $\mathbb{R}^2$ and $r\in \mathbb{R}$ we write $rS = \{rs~:~s \in S\}$.

\begin{theorem}[Cox and Durrett]\label{thm:limitshape}
For each $\mu \in \mathcal{M}$, there exists a deterministic, convex, compact set $B_\mu$ in $\mathbb{R}^2$ such that for each $\varepsilon>0$,
\[
\mathbb{P}\left( (1-\varepsilon)B_\mu \subset \frac{B(t)}{t} \subset (1+\varepsilon)B_\mu \text{ for all large } t \right) = 1\ .
\]
\end{theorem}

Although the above theorem gives the existence of a limit shape for any edge distribution,  no explicit expression is known for $B_\mu$ for any non-trivial measure $\mu \in \mathcal{M}$. 
This observation leads to fundamental questions for the study of these models:
\begin{enumerate}
\item Which compact convex sets $C$ are realizable as limit shapes?
\item What is the relationship between the measure $\mu$ and the limit shape $B_\mu$?
\item What are the fluctuations of $B(t)$ around its mean and around the set $tB_\mu$?
\end{enumerate}
In this paper,  we will mostly focus on the first question, though our proofs shed light on the second and we derive a few results related to the third one. Before describing our results, we emphasize that very little is known about these questions for FPP with i.i.d. passage times.  It is believed that under very general assumptions on $\mu$, the limit shape is strictly convex except in a very explicit region known as the {\it percolation cone}. It appears, however, that this statement is far from being proved. In fact, up until \cite{DH} there was not even a proven example of a measure which has a limit shape that is not a polygon. In this paper we give a major step in this direction.

We will first show that in a large subclass $\mathcal{M}_p$ of $\mathcal{M}$ (defined in the next section) the limit shapes have certain non-trivial points at which the boundaries are differentiable. We will then prove that this fact not only implies that the limit shapes are non-polygonal (extending \cite{DH}) but it also has a variety of consequences, for example for related growth models, and shape fluctuation estimates and exponents.

A word of comment is needed here. Interestingly, question 1 is completely solved by H\"aggstr\"om and Meester in the case of stationary passage times. In \cite{HM}, they show that given any compact convex set $C$ containing the origin which is also symmetric about the axes there exists a stationary (that is, invariant under lattice translations) ergodic measure $\mathbb{P}$ such that the limit shape for FPP with weights distributed by $\mathbb{P}$ is $C$. This is in sharp contrast with the i.i.d. case explained above.

\subsection{Flat edges for limit shapes}

Now we describe the class of measures under consideration. This collection was introduced by Durrett and Liggett \cite{DurrettLiggett} and further studied by Marchand \cite{Marchand}. Its main feature is the presence of a flat edge for the limit shape, as we describe below.

Write $supp(\mu)$ for the support of the measure $\mu$. Let $\mathcal{M}_p$ be the set of measures $\mu$ that satisfy the following:
\begin{enumerate}
\item $supp(\mu) \subseteq [1,\infty)$ and
\item $\mu(\{1\})=p\geq\vec p_c$,
\end{enumerate}
where $\vec p_c$ is the critical parameter for oriented percolation on $\mathbb{Z}^2$ (see, e.g., \cite{Durrett}). In \cite{DurrettLiggett}, it was shown that if $\mu \in \mathcal{M}_p$ for $ p > \vec p_c$, then $B_\mu$ has some flat edges. The precise location of these edges was found in \cite{Marchand}. To describe this, write $\mathcal{B}_1$ for the closed $\ell^1$ unit ball:
\[
\mathcal{B}_1 = \{(x,y)\in \mathbb{R}^2~:~ |x|+|y|\leq 1\}
\]
and write $int ~\mathcal{B}_1$ for its interior. For $p > \vec p_c$ let $\alpha_p$ be the asymptotic speed of oriented percolation \cite{Durrett}, define the points
\begin{equation}\label{eq:NP}
M_p = \left(\frac{1}{2} - \frac{\alpha_p}{\sqrt 2}, \frac{1}{2} + \frac{\alpha_p}{\sqrt 2}\right) \text{ and } N_p = \left(\frac{1}{2} + \frac{\alpha_p}{\sqrt 2}, \frac{1}{2} - \frac{\alpha_p}{\sqrt 2}\right)
\end{equation}
and let $[M_p,N_p]$ be the line segment in $\mathbb{R}^2$ with endpoints $M_p$ and $N_p$. For symmetry reasons, the following theorem of Marchand is stated only for the first quadrant.

\begin{theorem}[Marchand]\label{thm:marchand1}
Let $\mu \in \mathcal{M}_p$.
\begin{enumerate}
\item $B_\mu \subset \mathcal{B}_1$.
\item If $p < \vec p_c$ then $B_\mu \subset int~\mathcal{B}_1$.
\item If $p > \vec p_c$ then $B_\mu \cap [0,\infty)^2 \cap \partial \mathcal{B}_1 = [M_p,N_p]$.
\item If $p = \vec p_c$ then $B_\mu \cap [0,\infty)^2 \cap \partial \mathcal{B}_1 = (1/2,1/2)$.
\end{enumerate}
\end{theorem}

The angles corresponding to points in the line segment $[M_p,N_p]$ are said to be in the \emph{percolation cone}; see Figure~\ref{fig:pcone}. The reason for this term will become clear in Section~\ref{sec:Durr}.

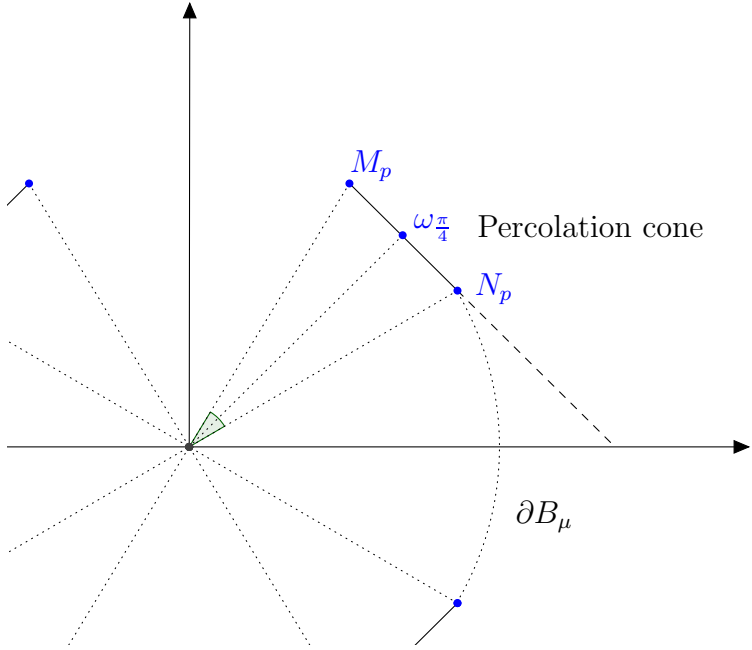
\begin{figure}\label{fig:pcone}
\centering
\definecolor{qqwuqq}{rgb}{0,0.39,0}
\definecolor{uququq}{rgb}{0.25,0.25,0.25}
\definecolor{qqqqff}{rgb}{0,0,1}
\begin{tikzpicture}[line cap=round,line join=round,>=triangle 45,x=0.7cm,y=0.7cm]
\clip(-3.42,-3.74) rectangle (12.97,9.55);
\draw [shift={(0,0)},color=qqwuqq,fill=qqwuqq,fill opacity=0.1] (0,0) -- (30.47:0.77) arc (30.47:58.93:0.77) -- cycle;
\draw (3.01,4.99)-- (5.04,2.96);
\draw [line width=0.4pt,dotted] (3.01,4.99)-- (0,0);
\draw [line width=0.4pt,dotted] (5.04,2.96)-- (0,0);
\draw (5.21,4.61) node[anchor=north west] {Percolation cone};
\draw[color=qqqqff] (4.01,4.61) node[anchor=north west] {$\omega_{\frac{\pi}{4}}$};
\draw (-3.01,4.99)-- (-5.04,2.96);
\draw [line width=0.4pt,dotted] (-3.01,4.99)-- (0,0);
\draw [line width=0.4pt,dotted] (4.01,4.01)-- (0,0);
\draw [line width=0.4pt,dotted] (-5.04,2.96)-- (0,0);
\draw [line width=0.4pt,dotted] (5.04,-2.96)-- (0,0);
\draw [line width=0.4pt,dotted] (3.01,-4.99)-- (0,0);
\draw (3.01,-4.99)-- (5.04,-2.96);
\draw [line width=0.4pt,dotted] (-5.04,-2.96)-- (0,0);
\draw [line width=0.4pt,dotted] (-3.01,-4.99)-- (0,0);
\draw (-3.01,-4.99)-- (-5.04,-2.96);
\draw [shift={(0,0)},dotted]  plot[domain=-0.53:0.53,variable=\t]({1*5.83*cos(\t r)+0*5.83*sin(\t r)},{0*5.83*cos(\t r)+1*5.83*sin(\t r)});
\draw [dash pattern=on 3pt off 3pt] (5.04,2.96)-- (8,0);
\draw [->] (0,0) -- (0.01,8.43);
\draw [->] (-4,0) -- (10.54,0);
\draw (5.93,-0.82) node[anchor=north west] {$\partial B_{\mu}$};
\fill [color=qqqqff] (3.01,4.99) circle (1.5pt);
\fill [color=qqqqff] (4.01,4.01) circle (1.5pt);
\draw[color=qqqqff] (3.42,5.33) node {$M_p$};
\fill [color=qqqqff] (5.04,2.96) circle (1.5pt);
\draw[color=qqqqff] (5.72,3) node {$N_p$};
\fill [color=uququq] (0,0) circle (1.5pt);
\fill [color=qqqqff] (-3.01,4.99) circle (1.5pt);
\fill [color=uququq] (0,0) circle (1.5pt);
\fill [color=qqqqff] (-5.04,2.96) circle (1.5pt);
\fill [color=uququq] (0,0) circle (1.5pt);
\fill [color=qqqqff] (5.04,-2.96) circle (1.5pt);
\fill [color=uququq] (0,0) circle (1.5pt);
\fill [color=qqqqff] (3.01,-4.99) circle (1.5pt);
\fill [color=uququq] (0,0) circle (1.5pt);
\fill [color=qqqqff] (3.01,-4.99) circle (1.5pt);
\fill [color=qqqqff] (5.04,-2.96) circle (1.5pt);
\fill [color=qqqqff] (-5.04,-2.96) circle (1.5pt);
\fill [color=uququq] (0,0) circle (1.5pt);
\fill [color=uququq] (0,0) circle (1.5pt);
\fill [color=qqqqff] (-4,0) circle (1.5pt);
\end{tikzpicture}
\caption{Pictorial description of the limit shape $B_\mu$. If $\mu \in \mathcal{M}_p$, the shape has a flat edge between the points $M_p$ and $N_p$. Outside the percolation cone the limit shape is unknown.}
\end{figure}

\section{Main Results}\label{sec: mainresults}

\subsection{Differentiability of the limit shape at percolation points}
The following is the main theorem of the paper: measures in $\mathcal{M}_p$ admit limit shapes whose boundaries are differentiable at the endpoints of the flat edge from Theorem~\ref{thm:marchand1}.

\begin{theorem}\label{thm: diff}
Let $\mu \in \mathcal{M}_p$ for $p\in[\vec p_c,1)$. The boundary $\partial B_\mu$ is differentiable at the point $N_p$.
\end{theorem}

\begin{remark}
Theorem~\ref{thm: diff} and Corollary~\ref{cor:nonpolygonal} below are stated for the single point $N_p$  but, due to symmetry, they are valid for $M_p$ and for the reflections of these two points about the coordinate axes.   
\end{remark}

As a consequence of this theorem, the limit shape must be non-polygonal. To state this precisely, let $ext(B_\mu)$ be the set of extreme points of $B_\mu$ and let $sides(B_\mu)$ be the number of points in $ext(B_\mu)$, so that $sides(B_\mu) <\infty$ if and only if $B_\mu$ is a polygon.

\begin{corollary}\label{cor:nonpolygonal}
If $\mu \in \mathcal{M}_p$ for $p\in[\vec p_c,1)$ then $sides(B_\mu) = \infty$ and $N_p$ is an accumulation point of $ext(B_\mu)$.
\end{corollary}
\begin{proof}
If $B_\mu$ had only finitely many extreme points, there would be a closest one to the point $N_p$. It follows that the arc of $\partial B_\mu$ from $N_p$ to this point is a line segment. But Theorem~\ref{thm: diff} (and also symmetry of $B_\mu$ about the line $y=x$ in the case $p=\vec p_c$) implies that this line segment must be contained in the set $\{(x,y) ~:~ x+y=1\}$. This contradicts Marchand's theorem.
\end{proof}

\begin{remark}\label{othercases} In Remark~\ref{rem: othercases}, we show that the analogue of Theorem \ref{thm: diff} also holds in the cases of (a) \emph{directed} first-passage percolation, (b) FPP with weights on sites (instead of edges) and (c)   directed FPP with weights on sites. To deduce the analogue of Corollary \ref{cor:nonpolygonal} in these cases, one needs a similar version of Theorem \ref{thm:marchand1}. The directed case (a) was studied in \cite{Zhang3}. \end{remark}

\begin{remark} After completion of the paper, the authors discovered that a claim was made after the statement of Theorem 1.4 of \cite{Marchand} which would have implied differentiability of the boundary of the limit shape at $N_p$. This claim was, however, unproved. 
\end{remark}

For the rest of Section~\ref{sec: mainresults} we will describe the consequences of Theorem~\ref{thm: diff}.

\subsection{Growth models}

First-passage percolation is closely related to certain growth and competition models. In fact the original version of the shape theorem was proved by Richardson \cite{Richardson} for a simplified growth model now known as the original 1-type Richardson model. In this model on $\mathbb{Z}^2$, we suppose that the origin houses an infection at time $t=0$ and all other sites are healthy. At each subsequent time $t=1, 2, \ldots$, any healthy site with at least one infected neighbor becomes infected independently of all other sites with probability $p \in (0,1)$. Richardson proved a shape theorem for the infected region at time $t$ as $t \to \infty$ and believed, on the basis of computer simulations, that as $p$ varies from $1$ to $0$, the limit shape varies from a ``diamond to a disk.'' This model was shown to be equivalent to an FPP model with i.i.d. weights on sites with a geometric distribution \cite{DurrettLiggett}.

We can build an infection model based on edge FPP in an analogous manner. Let $(\tau_e)$ be a realization of passage times. We infect the origin at time $0$ and the infection spreads at unit speed across edges, taking time $\tau_e$ to cross the edge $e$. In the case that the edge weights are exponential, the memoryless property of the distribution implies that the growth of the infected region is equal in distribution to a time change of the so-called Eden process (introduced in \cite{Eden}). It is also known as the 1-type Richardson model. For general distributions this growth process is called a 1-type first-passage percolation model.

Building on the above definition, we may describe first-passage competition models, first defined by H\"aggstr\"om and Pemantle \cite{HP}. Fix $k$ distinct vertices $x_{1},\ldots,x_{k}\in\mathbb{Z}^{2}$ and at time $t=0$ infect site $x_i$ by an infection of type $i$. Each species spreads at unit speed as before, taking time $\tau_{e}$
to cross an edge $e\in\mathbb{E}$. An uninhabited site is exclusively
and permanently colonized by the first species that reaches it; that is,
$y\in\mathbb{Z}^{2}$ is occupied at time $t$ by the $i$-th species
if $\tau(y,x_{i})\leq t$ and $\tau(y,x_{i})<\tau(y,x_{j})$ for all
$j\neq i$.  Note that there may be sites which are never colonized, that is, those sites $y$ for which $\min_{1\leq i\leq k}\tau(y,x_i)$ is achieved by multiple $x_i$'s. Consider the set colonized
by the $i$-th species:\[
C_{i}=\{y\in\mathbb{Z}^{2}\,:\, y\mbox{ is eventually occupied by }i\}.\]
One says that $\mu$ \emph{admits coexistence of $k$ species} if for some choice of initial sites
$x_{1},\ldots,x_{k}$, \[
\mathbb{P}(|C_{i}|=\infty\mbox{ for all }i=1,\ldots,k)>0.\]
Coexistence of infinitely many species is defined similarly and clearly implies coexistence of $k$ species for any $k$.

In the past ten years, there has been a growing interest in Richardson-type models, for instance in questions related to the asymptotic shape of infected regions \cite{Gouere, Pimentel} and to coexistence \cite{GM, HP, Hoffman, Hoffman1}. When $\mu$ is the exponential distribution, H\"{a}ggstr\"{o}m and Pemantle \cite{HP} proved coexistence of 2 species (see \cite{DH} for a review of recent results on Richardson models, focused on exponential passage times). Shortly thereafter, Garet and Marchand \cite{GM} and Hoffman \cite{Hoffman1} independently extended these results to prove coexistence of 2 species for a broad class of translation invariant measures, including some non-i.i.d. ones.  Later, Hoffman \cite{Hoffman} demonstrated coexistence of $8$ species for a similarly broad class of measures by establishing a relation with the number of
sides of the limit shape in the associated FPP. In \cite{DH}, Damron and Hochman showed that it is possible to have infinite coexistence. As a consequence of Theorem~\ref{thm: diff}, infinite coexistence is not only possible but is necessary for all measures in $\mathcal{M}_p$.

\begin{theorem}\label{thm:coexistence}
If $\mu \in \mathcal{M}_p$ for $p\in[\vec p_c,1)$ then $\mu$ admits coexistence of infinitely many species.
\end{theorem}

\begin{proof}
Hoffman showed in \cite{Hoffman} that coexistence of $k$ species occurs so long as $sides(B_\mu)\geq k$. His arguments were extended in \cite{DH} to cover the case of infinitely many species. The theorem then follows from  Corollary~\ref{cor:nonpolygonal}.
\end{proof}

To obtain more information about the spread of infections, we may track the edges that an infection crosses from its starting point. Accordingly, we define the {\it graph of infection} $\Gamma(0)\subseteq\mathbb{E}$ as the union over all
$x\in\mathbb{Z}^d$ of the edges of geodesics from $0$ to $x$. This terminology is consistent with the Richardson model when there are unique passage times, in which case it is a tree, but note that in general the graph of infection may also contain sites which were not infected; that is, those where a tie condition exists, and in this way we may obtain loops.
A graph has $m$ {\it ends} if, after removing a finite set of vertices, the induced graph contains at least $m$ infinite connected components, and, if there are $m$ ends for every $m\in\mathbb{N}$, we say there are infinitely many ends. Letting $K(\Gamma(0))$ be the number of ends in $\Gamma(0)$, Newman \cite{Newman} has conjectured for a broad class of $\mu$ that $K(\Gamma(0))=\infty$. Hoffman \cite{Hoffman} showed for continuous distributions that in general $K(\Gamma(0)) \geq 4$ almost surely. Damron and Hochman \cite{DH} then proved that there exist $\mu$ such that the graph of infection has infinitely many ends. Using their arguments, we have the following.

\begin{theorem}\label{thm:ends}
Let $\mu \in \mathcal{M}_p$ for $p\in[\vec p_c,1)$ be not purely atomic. Then almost-surely, $K(\Gamma(0)) = \infty$.
\end{theorem}

\begin{proof}
Theorem~\ref{thm:ends} follows directly from the following relation established in \cite{DH}. If $\mu$ is not purely atomic and $B_\mu$ has at least $s$ sides, then the number of ends of $\Gamma(0)$ is almost-surely at least $4 \lfloor (s-4)/12 \rfloor$.
\end{proof}

\begin{remark}
In the case that $\mu \in \mathcal{M}_p$ is purely atomic, it is unknown and, furthermore, unclear whether or not $K(\Gamma(0))=\infty$. The method of \cite{DH} heavily relies on the existence of a continuous part of the edge distribution. 
\end{remark}

\subsection{Fluctuations}

We now proceed to analyze the fluctuations of the limit shape outside the percolation cone. 

The growth of $B(t)$ in first-passage percolation as $t \to \infty$ gives a simple model for the growth of a random interface. The leading order behavior of this growth is given by the shape theorem. When inspecting second order behavior, we can begin with dimension $d=1$. In this case, $\tau(0,n)$ is equal to $\tau(\gamma)$, where $\gamma$ is the deterministic path with vertices $0,1,\ldots, n$. Therefore $\tau(0,n)$ is a sum of i.i.d. variables and will obey diffusive scaling (that is, after centering and scaling by $\sqrt n$, it will converge in distribution to a non-trivial law). This same scaling is proved to be correct in thin cylinders \cite{CD}. 

In two and higher dimensions, the passage time of each deterministic path $\gamma$ is still a sum of i.i.d. random variables, but the passage time from $0$ to a vertex $v$ is a minimum over these correlated path variables. Because of the correlation structure, physicists predict \cite{KPZ} that the centered variable $\tau(0,ne_1)$, where $e_1$ is the first coordinate vector, should not obey diffusive scaling, but a certain sub-diffusive scaling (that is, the appropriate exponent of $n$ is $1/3$, not $1/2$). 

Relatively very little is known rigorously in this direction. Excluding an exactly solvable case in a related model where the $1/3$ exponent is known and one even has an explicit form for the limiting distribution \cite{Johan}, one does not even have a lower bound of the form $n^\alpha$ for some $\alpha>0$. The first substantial lower bounds were given by Pemantle and Peres \cite{PP} and Newman and Piza \cite{NewmanPiza}. In the latter work it is shown that if $\mu$ has the properties (a) if $\inf~supp(\mu) = 0$ and $\mu(\{0\})<p_c$ (the threshold for unoriented percolation) and (b) if $b=\inf~supp(\mu)$ and $\mu(\{b\})<\vec p_c$ (the threshold for oriented percolation) then for each non-zero $x\in \mathbb{Z}^2$, there exists $C_x>0$ such that for all $n$, $\Var \tau(0,nx) \geq C_x \log n$. Newman and Piza state that for $\mu \in M_p$ with $p > \vec p_c$, the statement is not true for $x$ in the direction of the percolation cone. However it may be true for other angles. Zhang extended the theorem to measures in $\mathcal{M}_p$ in the $x$-direction. The following is the main theorem of \cite{Zhang}. 

\begin{theorem}[Zhang]\label{thm:Zhang}
Let $\mu \in \mathcal{M}_p$ for $p \in [\vec p_c,1)$ and suppose that $\E e^{\tau_e\beta}<\infty$ for some $\beta>0$. There exists $C>0$ such that for all $n$,
\[
\Var~\tau(0,(n,0)) > C \log n\ .
\]
\end{theorem}
\noindent
Zhang also predicts (see Remark~6 in \cite{Zhang}) that a similar bound holds in any direction outside the percolation cone. In Section \ref{sec:Newman}, we verify this prediction, extending Zhang's theorem to all angles outside the percolation cone assuming exponential moments for the edge distribution. This is Corollary~\ref{cor: exp} below. 

To state these results, let $\theta_p$ be the unique angle such that the line segment connecting $0$ and $N_p$ has angle $\theta_p$ with the $x$-axis. Let $w_\theta$ be the vector 
\[
w_\theta = (\cos \theta, \sin \theta)
\]
and for fixed $x \in \mathbb{R}^2$, define 
\[
g(x) = \lim_{n \to \infty} (1/n) \E \tau(0,nx)\ .
\]
This limit exists by the shape theorem. (We will explain more about the function $g$ in Section~\ref{sec:Newman}, for instance that it is a natural norm on $\mathbb{R}^2$ associated to the model. For the moment, however, the reader may ignore the definition of $g$ and the condition \eqref{eq:sad} which uses it and instead focus on Corollaries~\ref{cor: exp} and \ref{cor: beta}, whose assumptions allow to remove this condition.) We begin with a variance lower bound that requires only two moments for the edge distribution but a bound \eqref{eq:sad} on the so-called ``non-random fluctuations'' of the model. 
\begin{theorem}\label{thm:unbounded}
Let $\mu \in \mathcal{M}_p$ for $p\in[\vec p_c,1)$ and $\theta \in [0,\theta_p)$. Suppose that $\E \tau_e^2<\infty$ and that there exists $\kappa < 1$ such that for all large $n$, 
\begin{equation}\label{eq:sad}
\E\tau(0,nw_\theta) - ng(w_\theta) < n^{\kappa} \ .
\end{equation}
Then there exists $C_\theta>0$ such that for all $n$,
\begin{equation}\label{eq:loglowerbound}
\Var~\tau(0,nw_\theta)\geq C_\theta \log n\ .
\end{equation}
\end{theorem}

The assumption \eqref{eq:sad} is known in two cases: (a) if $\E e^{\tau_e \beta} < \infty$ for some $\beta > 0$ with any $\theta$ (Theorem~3.2 of \cite{Alexander}) and (b) if $\E \tau_e^{1+\beta}<\infty$ for some $\beta >0$ and $\theta = 0$ (Theorem~3 of \cite{Zhang2}).  Therefore we can state these special cases. The first shows the logarithmic lower bound is valid for all directions outside the percolation cone and the second allows to reduce the moment condition in Zhang's Theorem~\ref{thm:Zhang}. 

\begin{corollary}\label{cor: exp}
Let $\mu \in \mathcal{M}_p$ for $p\in[\vec p_c,1)$ and suppose that $\E e^{\tau_e\beta} <\infty$ for some $\beta >0$.  Then for all directions $\theta \in [0,\theta_p)$, there exists $C_\theta >0$ such that for all $n$, \eqref{eq:loglowerbound} holds.
\end{corollary}

\begin{corollary}\label{cor: beta}
Let $\mu \in \mathcal{M}_p$ for $p\in[\vec p_c,1)$ and suppose that $\E \tau_e^2 <\infty$.  Then there exists $C_0>0$ such that for all $n$, \eqref{eq:loglowerbound} holds with $\theta = 0$.
\end{corollary}

\begin{remark} Inside the percolation cone, the behavior of the variance of the passage time is different from \eqref{eq:loglowerbound}. For any $\theta \in (\theta_p, \frac{\pi}{4})$, it is known that  $\Var \; \tau(0,n\omega_\theta)$ is bounded above independently of $n$ (see special case 2 of \cite{NewmanPiza} or (1.11) in \cite{Zhang}).  
\end{remark}

For the other theorems, we need to define two exponents. We shall use the definitions from \cite{NewmanPiza}. Let
\begin{equation}\label{eq:chitheta}
\chi_\theta = \sup \{\gamma\geq 0~:~ \text{ for some } C>0, \Var(\tau(0,nw_\theta)) \geq Cn^{2\gamma} \text{ for all } n\}
\end{equation}
and let $M_n(\theta)$ be the set of all sites in $\mathbb{Z}^2$ belonging to some geodesic from $0$ to $nw_\theta$. Let $L_\theta$ be the line through the origin containing $w_\theta$ and for $\gamma >0$ let $\Lambda_n^\gamma(\theta)$ be the set
\[
\Lambda_n^\gamma(\theta) = \{ z \in \mathbb{R}^d~:~ dist(z,L_\theta)\leq n^\gamma\}\ .
\]
The exponent $\xi_\theta$ is defined as
\begin{equation}\label{eq:xitheta}
\xi_\theta = \inf\{\gamma>0~:~ \text{ for some } C>0, ~\mathbb{P}(M_n(\theta) \subset \Lambda_n^\gamma)\geq C \text{ for all large } n\}\ .
\end{equation}

\begin{theorem}\label{thm:powerlaw}
If $\mu \in \mathcal{M}_p$ for $p \in [\vec p_c,1)$ and $\E \tau_e^2<\infty$, then for each $\theta \in [0,\theta_p)$,
\[
\chi_\theta \geq \frac{1-\xi_\theta}{2}\ .
\]
\end{theorem}

Using a curvature assumption, we may obtain a theorem that is much stronger than Theorem~\ref{thm:unbounded}. This assumption was used in \cite{NewmanPiza} to show power-law divergence of the passage-time variance in all directions for a different class of edge-weight distributions than ours. For any angle $\theta$, we say that $\theta$ is a {\it direction of curvature} if there exists a closed Euclidean ball $B$ of positive radius (and any center) such that (a) $\partial B_\mu$ intersects $B$ at least on the line $L_\theta$ and (b) $B$ contains $B_\mu$. 

\begin{theorem}
Let $\mu \in \mathcal{M}_p$ for $p \in [\vec p_c,1)$. If $\mathbb{E} e^{\tau_e\beta}< \infty$ for some $\beta>0$ and $\theta \in [0,\theta_p)$ is a direction of curvature, then there exists $C_\theta>0$ such that for all $n$,
\begin{equation}\label{eq:powerlaw}
\Var \tau(0,nw_\theta)\geq C_\theta n^{1/4}\ .
\end{equation}
\end{theorem}

\begin{proof}
This follows just as in \cite{NewmanPiza} by using the upper bound on shape fluctuations from Kesten \cite{Kesten} and Alexander \cite{Alexander}, along with \cite[Theorem~7]{NewmanPiza}. All of these theorems are valid for measures in $\mathcal{M}_p$.
\end{proof}

\begin{remark}
Using the construction of \cite{NewmanPiza}, let $B^r$ be the Euclidean ball of radius $r$ centered at the origin and set $\rho = \inf\{r>0~:~B_\mu \subset B^r\}$. The intersection points of $\partial B^\rho$ and $\partial B_\mu$ are in directions of curvature of $B_\mu$. From Theorem~\ref{thm: diff}, given $p \geq \vec p_c$ and $\mu \in \mathcal{M}_p$, the boundary $\partial B_\mu$ is differentiable at $M_p$ and $N_p$ (and at their reflections about the coordinate axes). This implies that for $p>\vec p_c$ these points cannot be in directions of curvature. Combined with the fact that no direction in the percolation cone can be a direction of curvature, we see that for these measures (assuming the exponential moment condition), there exists $\theta \in [0,\theta_p)$ such that \eqref{eq:powerlaw} holds.
\end{remark}

The rest of the paper is organized as follows. In Section~\ref{sec:thmdiff} we prove Theorem~\ref{thm: diff}. Next, in Section \ref{sec:Newman} we prove Theorem~\ref{thm:unbounded} and we describe why the initial work of \cite{NewmanPiza} was not able to cover the case treated in this paper. In the last section, we prove Theorem~\ref{thm:powerlaw}.

\section{Proof of Theorem~\ref{thm: diff}}\label{sec:thmdiff}

In this part of the paper we prove Theorem~\ref{thm: diff}. For the reader's convenience we will  divide it into small subsections as follows. In the first part we give a brief idea about how we will proceed to prove the theorem. After that, we introduce an embedded oriented percolation model and an essential lemma adapted from Durrett \cite{Durrett} that will allow our construction of special paths called bypasses. In Section~\ref{subsec:bypass}, we give the inductive construction of the bypass paths and in the last two sections we prove Theorem~\ref{thm: diff} in the cases $p> \vec p_c$ and $p = \vec p_c$, respectively.

\subsection{Strategy of the proof of Theorem~\ref{thm: diff}.}

Our goal in this section is to show that if $\mu \in \mathcal{M}_p$ then the limit shape boundary, $\partial B_\mu$, is differentiable at the point $N_p$. For this purpose, fix a linear functional $f$ such that (a) $\max_{z \in B_\mu} f(z) = 1$ and (b) $f(N_p)=1$. By convexity of $B_\mu$, the line 
\[
L_1(f) = \{(x,y) \in \mathbb{R}^2~:~f(x,y) =1\}
\]
is tangent to $B_\mu$ at $N_p$. 

Let us make an assumption that the line $\overline L = \{(x,y) \in \mathbb{R}^2~:~x+y=1\}$ is not the only tangent line to $\partial B_\mu$ at the point $N_p$. Therefore we may choose $f$ such that it is strictly increasing on this line, that is, for fixed $z \in \mathbb{R}^2$, the function
\begin{equation}\label{eq: teq}
t \mapsto f(z+t(1,-1)) \text{ is strictly increasing in } t\ .
\end{equation}
In other words, 
\begin{equation}\label{defCf} 
C_f := f(1,-1)>0.
\end{equation}
 We will see later that this assumption gives a contradiction. We will then conclude that the only tangent line to $B_\mu$ at the point $N_p$ is $\overline L$ and therefore $\partial B_\mu$ is differentiable at this point.

The main idea is that if $L_1$ is a tangent line to $B_\mu$ then for very large $n$, the optimal path from $0$ to $nL_1$ will, up to linear order in $n$, have the same travel time as a specific {\it maximal} path constructed from an oriented percolation process. We will show that this leads to a contradiction because we can make modifications of such a maximal path which take advantage of small sets of edges called {\it $f$-bypasses}, and which give a faster route from $0$ to $nL_1$. These $f$-bypasses are constructed depending on the angle between the lines $L_1$ and $\overline L$, and the amount of time saved by taking them is strictly positive on a linear scale. Our proof is inspired by the technique and original ideas introduced by Marchand in the proof of \cite[Theorem~1.4]{Marchand}, whose notation we will follow. In addition, an important lemma comes from Durrett's proof \cite{Durrett} of strict monotonicity of the mapping $p \mapsto N_p$.

\medskip

\subsection{Coupling with Oriented Percolation}\label{sec:Durr}

The existence of a flat edge for measures in $\mathcal{M}_p$ is related to a coupling between first-passage percolation and {\it oriented percolation} that we describe now. If we call edges $e$ with $\tau_e=1$ {\it open} and all other edges {\it closed}, then for $p>\vec p_c$, there exists an infinite oriented cluster of open edges \cite{Durrett}. This means that with positive probability, an infinite number of vertices $v$ can be reached from the origin using paths that only move up and right, and whose edges have passage time equal to $1$. In fact, each such path must be a geodesic and this implies that the passage time from $0$ to $v$ is simply equal to the $\ell^1$-distance from $0$ to $v$. The set of vertices thus reached is called the {\it oriented percolation cone} and we will say that an angle $\theta$ is in the oriented percolation cone if there is an infinite oriented path of 1-edges $(x_1,x_2, \ldots)$ such that the argument (angle) of $x_n$ with respect to the $x$-axis converges to $\theta$. Marchand's theorem states that the point of $\partial B_\mu$ in the direction $\theta$ is on the boundary of the $\ell^1$ unit ball if and only if $\theta$ is in the percolation cone.

Precisely, for each realization of passage times $(\tau_e)$, we define edge variables $(\eta_e)$ by the rule that $\eta_e=1$ if $\tau_e=1$ and $\eta_e=0$ if $\tau_e\neq 1$. We write $x \to y$ if there exists an oriented path from $x$ to $y$ with all edges having $\eta$-value equal to $1$.

%
We call a  subset $S$ of $\mathbb{Z}^2$ a \emph{starting set} if it is contained in some integer translate of the one sided ray $\widetilde{\mathbb{N}} := \{(-k,k)~:~k \in \mathbb{N}\}$.  Write
\[
D_n = \{(x,y) \in \mathbb{Z}^2~:~ x+y=n\}\ .
\]
For each $n \in \mathbb{N}$, we define the collection $\xi_n(S)$ of accessible diagonal sites from a starting set $S \subseteq D_m$ as the set of vertices $x \in D_{m+n}$ such that there exists $s\in S$ with $s \to x$.

 If $S$ is a starting set then for each $n \in \mathbb{N}$, $M_n(S)$ is defined to be the point of $\xi_n(S)$ with largest $x$-coordinate. If $\xi_n(S) = \varnothing$ we leave $M_n(S)$ undefined. We also set $f_n(S) = f(M_n(S))$ in the case that $M_n(S)$ is defined; otherwise, we set $f_n(S) = -\infty$. Note that by \eqref{eq: teq}, if $S$ is a starting set then $M_n(S)$ is the point of $\xi_n(S)$ with strictly largest $f$-value.

Next we prove a lemma about the oriented percolation process with edge variables $(\eta_e)$. It is adapted from Durrett \cite[Equation~(13)]{Durrett}.

\begin{lemma}\label{lem: Durrett}
 If $A$ and $B$ are infinite subsets of $\widetilde{\mathbb{N}}$ with $A \subset B$ then for all $p \in (0,1)$, all $n$ and all $m \geq 1$,
\[
\E \left[ f_n(A\cup E_m) - f_n(A) \right] \geq \E \left[ f_n(B\cup E_m) - f_n(B) \right] \geq C_f m\ ,
\]
where $E_m = \{(k,-k)~:~1 \leq k \leq m\}$ and $C_f$ is the constant defined in \eqref{defCf}.
\end{lemma}

\begin{proof}
Clearly for sets $C$ and $D$ in $\widetilde{\mathbb{N}}$ (finite or infinite) we have $\xi_n(C\cup D)= \xi_n(C) \cup \xi_n(D)$. Therefore by \eqref{eq: teq},
\[
f_n(C\cup D) = \max\{f_n(C),f_n(D)\}\ .
\]
Now if $f_n(C) > -\infty$ (that is, $\xi_n(C) \neq \varnothing$),
\begin{eqnarray*}
f_n(C\cup D) - f_n(C) &=& \max\{f_n(C),f_n(D)\} - f_n(C) \\
&=& \begin{cases}
0&\text{ if } f_n(C) \geq f_n(D) \\
f_n(D)-f_n(C)& \text{ if } f_n(C) < f_n(D)
\end{cases} \\
&=& (f_n(D)-f_n(C))_+\ .
\end{eqnarray*}
Here $r_+$ is the positive part of the real number $r$. If $A$ and $B$ are infinite,  then since $p>0$, $f_n(A)$ and $f_n(B)$ are finite almost surely for all $n$. By \eqref{eq: teq} and the assumption that $A \subset B$, 
\[
f_n(B\cup E_m) - f_n(B) = (f_n(E_m)-f_n(B))_+ \leq (f_n(E_m)-f_n(A))_+ = f_n(A \cup E_m) - f_n(A)\ .
\]
Taking expectation of both sides gives the first part of the lemma. For the second part we simply use translation invariance. By the first part, it suffices to prove it for $A= \widetilde{\mathbb{N}}$. Let $T$ be the translation of $\mathbb{Z}^2$ that maps the origin to the point $(1,-1)$. For a configuration $\omega$ of passage times, we define $T(\omega)$ by the rule
\[
\tau_e(T(\omega)) = \tau_{T^{-1}(e)}(\omega)\ .
\]
Clearly the measure on passage time configurations is invariant under $T$. Writing $f_n(A)(\omega)$ for the value of the variable $f_n(A)$ in the configuration $\omega$, we see that since $\xi_n(A) \neq \varnothing$ for all $n$ almost surely,
\[
f_n(A \cup E_m)(T^m(\omega)) - f_n(A)(\omega) = mf(1,-1) \text{ almost surely for all } n \text{ and } m \geq 1\ .
\]
Therefore
\[
\E\left[ f_n(A\cup E_m) - f_n(A)\right] = \E \left[f_n(A\cup E_m)(T^m(\omega)) - f_n(A)(\omega)\right] = mf(1,-1)\ .
\]
\end{proof}

\subsection{Construction of the $f$-bypasses}\label{subsec:bypass}

From now on, we fix an integer $a\geq 2$. We first define a `good configuration' around a vertex.

\begin{definition}
We say that there is the (C$_a$)-configuration around the vertex $(x,y)$ in $\mathbb{Z}^2$ if the following occurs.
\begin{itemize}
\item The edges $\langle(x,y),(x+1,y)\rangle, \langle(x+1,y),(x+1,y-1)\rangle$ and $\langle(x+1,y-1),(x+2,y-1)\rangle$ have $\eta$-value equal to 1, that is, they are open for the embedded oriented percolation model.
\item The edges $\langle(x,y),(x,y+1)\rangle, \langle(x+1,y),(x+2,y)\rangle$ and $\langle(x+1,y),(x+1,y+1)\rangle$ and have $\eta$-value equal to 0, that is, they are closed for  the embedded oriented percolation model.
\item For each vertex $w$ in the set $\{(u,v)~:~u+v=x+y \text{ and } y+1 \leq v \leq  y+a-2\}$, each edge with $w$ as either a left endpoint or a bottom endpoint has $\eta$-value equal to 0. (For $a=2$, this condition is not used.)
\item Each edge between pairs of nearest-neighbor vertices in the set 
\[
\{(u,v) ~:~u \geq x+2, ~v \geq y-1, \text{ and } \|(u,v)-(x+2,y-1)\|_1 \leq a-1\}
\]
has $\eta$-value equal to $1$. (Here $\|\cdot \|_1$ represents the $l^1$ norm.)
\end{itemize}
\end{definition}

\noindent
See Fig.~\ref{fig:Caconfig} for an illustration of the event that there is the (C$_a$)-configuration around the vertex $M_{an-a}^0$. 

\vspace{0.2cm}
Notice that the probability of a (C$_a$)-configuration around a point $(x,y)$ is positive independently of $(x,y)$. Call this probability $\rho_a$. 

We will use the (C$_a$)-configurations to construct our bypasses. First, let us define the set
\[
S^a = \{(u,v)\in \mathbb{Z}^2~:~u \geq 2,~v \geq -1, \text{ and }\|(u,v) -(2,-1)\|_1 = a-1\}\ .
\]
The cardinality of $S_a$ is equal to $a$.

\definecolor{red}{rgb}{1,0,0}
\definecolor{uququq}{rgb}{0.25,0.25,0.25}
\definecolor{xdxdff}{rgb}{0.49,0.49,1}
\definecolor{qqqqcc}{rgb}{0,0,0.8}
\definecolor{qqqqff}{rgb}{0,0,1}
\begin{figure}
\centering
\begin{tikzpicture}[line cap=round,line join=round,>=triangle 45,x=0.95cm,y=0.95cm]
\clip(-5.34,-2.1) rectangle (7.13,8.02);
\draw [dotted,domain=-5.34:7.13] plot(\x,{(--2-1*\x)/1});
\draw [dotted,domain=-5.34:7.13] plot(\x,{(--6-2*\x)/2});
\draw [dotted,domain=-5.34:7.13] plot(\x,{(--8-2*\x)/2});
\draw [dotted,domain=-5.34:7.13] plot(\x,{(--5-1*\x)/1});
\draw [dotted,domain=-5.34:7.13] plot(\x,{(--6-1*\x)/1});
\draw [dotted,domain=-5.34:7.13] plot(\x,{(--7-1*\x)/1});
\draw (0,2)-- (1,2);
\draw (1,2)-- (1,1);
\draw [dash pattern=on 2pt off 2pt] (1,2)-- (2,2);
\draw [dash pattern=on 2pt off 2pt] (1,2)-- (1,3);
\draw [dash pattern=on 2pt off 2pt] (0,2)-- (0,3);
\draw [dash pattern=on 2pt off 2pt] (0,3)-- (-1,3);
\draw [dash pattern=on 2pt off 2pt] (-1,3)-- (-1,4);
\draw [dash pattern=on 2pt off 2pt] (-1,4)-- (-2,4);
\draw [dash pattern=on 2pt off 2pt] (-2,5)-- (-2,4);
\draw [dash pattern=on 2pt off 2pt] (-3,5)-- (-2,5);
\draw (2,2)-- (5,2);
\draw (3,1)-- (3,2);
\draw (2,1)-- (2,2);
\draw (4,1)-- (4,2);
\draw (2,3)-- (2,2);
\draw (2,5)-- (2,3);
\draw (2,3)-- (4,3);
\draw (3,4)-- (3,2);
\draw (4,3)-- (4,2);
\draw (5,1)-- (5,2);
\draw (6,1)-- (1,1);
\draw (1,1)-- (2,1);
\draw [dash pattern=on 2pt off 2pt] (-3,6)-- (-3,5);
\draw (3,4)-- (2,4);
\draw [color=qqqqcc] (2,0)-- (6,0);
\draw [color=qqqqcc] (-3.84,6)-- (-3.84,2);
\draw [->] (-1.46,0.75) -- (-0.36,1.66);
\draw (-2.01,0.86) node[anchor=north west] {$\mathit{M_{a\tau_1-a}^0}$};
\draw (4.79,4.33) node[anchor=north west] {$\mathit{S_{a\tau_1}^0}$};
\draw (3.27,-0.1) node[anchor=north west] {$\mathit{a-1}$};
\draw (-5.1,4.15) node[anchor=north west] {$\mathit{a-1}$};
\draw [->,color=qqqqcc] (6.39,1.66) -- (2.94,5.09);
\draw [->,color=qqqqcc] (2.94,5.09) -- (6.39,1.66);
\fill [color=qqqqff] (1,1) circle (1.5pt);
\fill [color=qqqqff] (2,1) circle (1.5pt);
\fill [color=qqqqff] (3,1) circle (1.5pt);
\fill [color=qqqqff] (4,1) circle (1.5pt);
\fill [color=qqqqff] (5,1) circle (1.5pt);
\fill [color=qqqqff] (6,1) circle (1.5pt);
\draw [color=red] (0,2)-- ++(-2.5pt,-2.5pt) -- ++(7.0pt,7.0pt) ++(-7.0pt,0) -- ++(7.0pt,-7.0pt);
\fill [color=qqqqff] (1,2) circle (1.5pt);
\fill [color=qqqqff] (2,2) circle (1.5pt);
\fill [color=qqqqff] (3,2) circle (1.5pt);
\fill [color=qqqqff] (4,2) circle (1.5pt);
\fill [color=qqqqff] (5,2) circle (1.5pt);
\fill [color=qqqqff] (0,3) circle (1.5pt);
\fill [color=qqqqff] (1,3) circle (1.5pt);
\fill [color=qqqqcc] (-1,3) circle (1.5pt);
\fill [color=qqqqcc] (-1,4) circle (1.5pt);
\fill [color=qqqqcc] (-2,4) circle (1.5pt);
\fill [color=qqqqcc] (-2,5) circle (1.5pt);
\fill [color=qqqqcc] (-3,5) circle (1.5pt);
\fill [color=qqqqcc] (2,3) circle (1.5pt);
\fill [color=qqqqcc] (2,5) circle (1.5pt);
\fill [color=qqqqcc] (4,3) circle (1.5pt);
\fill [color=qqqqcc] (3,4) circle (1.5pt);
\fill [color=qqqqcc] (-3,6) circle (1.5pt);
\fill [color=qqqqcc] (2,4) circle (1.5pt);
\fill [color=qqqqcc,shift={(2,0)},rotate=90] (0,0) ++(0 pt,2.25pt) -- ++(1.95pt,-3.375pt)--++(-3.9pt,0 pt) -- ++(1.95pt,3.375pt);
\fill [color=qqqqff,shift={(6,0)},rotate=270] (0,0) ++(0 pt,2.25pt) -- ++(1.95pt,-3.375pt)--++(-3.9pt,0 pt) -- ++(1.95pt,3.375pt);
\fill [color=qqqqff,shift={(-3.84,6)}] (0,0) ++(0 pt,2.25pt) -- ++(1.95pt,-3.375pt)--++(-3.9pt,0 pt) -- ++(1.95pt,3.375pt);
\fill [color=qqqqff,shift={(-3.84,2)},rotate=180] (0,0) ++(0 pt,2.25pt) -- ++(1.95pt,-3.375pt)--++(-3.9pt,0 pt) -- ++(1.95pt,3.375pt);
\fill [color=qqqqcc] (3.02,2.98) circle (1.5pt);
\end{tikzpicture}
\caption{A depiction of the (C$_a$)-configuration around the vertex $M_{a\tau_1-a}^0$. The dashed bonds have $\eta$-value zero and the full bonds have $\eta$-value one.}
\label{fig:Caconfig}
\end{figure}
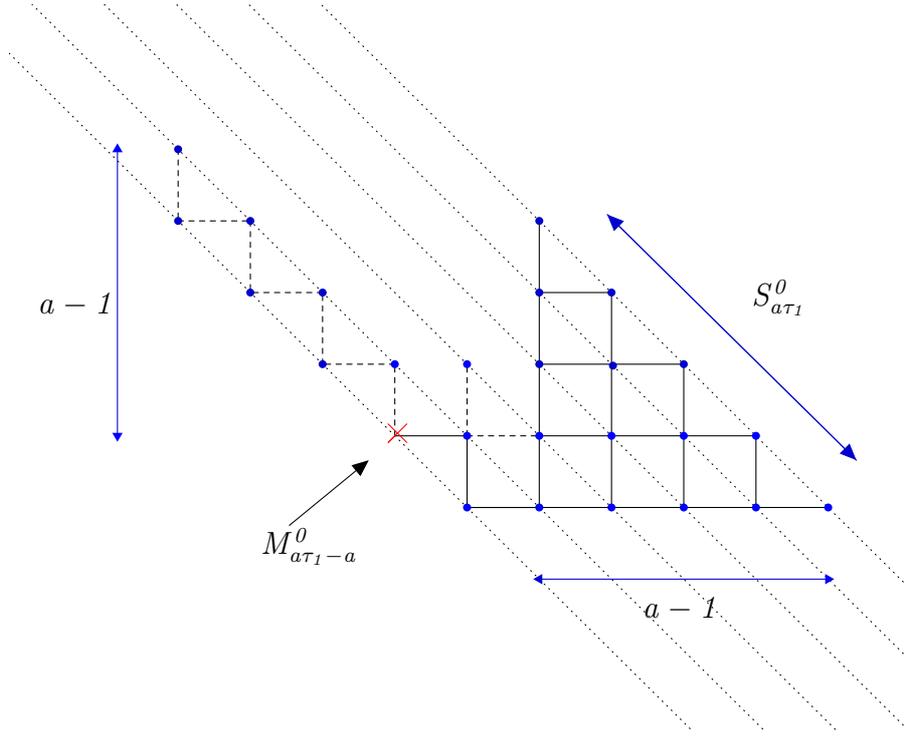

Let $S$ be an infinite starting set. To make the notation lighter we set $\xi_n^0=\xi_n(S)$ and note that since $p>0$, $\xi_n^0 \neq \varnothing$ almost surely for every $n$. Set $f_n^0=f_n(S)$ and let $M_n^0=M_n(S)$.  We will recall the dependency of these random variables on $S$ when it is necessary. Let $\tau_1=\tau_1(S)$ be the smallest $n$ such that we have a (C$_a$)-configuration around $M_{an-a}^0$. Notice that $\tau_1$ is a stopping time with respect to the filtration $(\mathcal{G}_n)$, where
\[
\mathcal{G}_n = \sigma(\{\tau_e : e \text{ has both endpoints in the halfspace } x+y\leq n\})\ .
\]

For each point $z \in \xi_{a\tau_1}^0$ there exists an (oriented) path with passage time equal to $a\tau_1$ from the set $S$ to $z$. Now, look at Figure \ref{fig:Caconfig}. For each point $z' \in M_{a\tau_1-a}^0 + S^a$, there exists a path (non-oriented) from $S$ to $z'$ with passage time equal to $a\tau_1+2$. It is important to notice that, because the (C$_a$)-configuration is defined specifically so that 
\begin{equation}\label{eq: twosets}
\{M_{a\tau_1-a}^0 + S^a\} \cap \xi_{a\tau_1}^0 = \emptyset\ ,
\end{equation}
these points $z'$ cannot be in the oriented cluster of open edges of $S$. The union of the two disjoint sets in \eqref{eq: twosets} is our new collection of reachable vertices and the first $f$-bypass paths are the open paths from $M_{a\tau_1-a}^0$ to the points in  $S_{a\tau_1}^0:= M_{a\tau_1-a}^0 + S_a$. Precisely, we set
\[
\xi_n^1 = \begin{cases}
\xi_n^0& \text{ if }n \leq a\tau_1-1 \\
\xi_{n-a\tau_1}(\xi_{a\tau_1}^0 \cup S_{a\tau_1}^0)& \text{ if }n \geq a\tau_1
\end{cases}\ .
\]

Now we define $M_n^1$ to be the point in $\xi_n^1$ with largest $x$-coordinate and
\[
f_n^1 = f_n^1(S) := f(M_{n}^1).\]

A crucial estimate on the  between $\E f_{an}^1$  and $\E f_{an}^0$ is the following (recall that these numbers depend on our initial choice of $S$):

\begin{lemma}\label{l:festimate} For all $n \in \N$ and any infinite starting set $S$,
\begin{equation*}
\E f_{an}^1(S) - \E f_{an}^0 (S) \geq aC_f \P(\tau_1(S) \leq n).
\end{equation*}
\end{lemma}
\begin{proof}
 We will apply Lemma~\ref{lem: Durrett}. By definition of $\xi_n^1$ and $\tau_1$, 
\begin{eqnarray}\label{eq: onestep}
\E f_{an}^1 - \E f_{an}^0 &=& \E \left[(f_{an}^1 - f_{an}^0) I(\tau_1 \leq n) \right] \nonumber\\
&=& \E \left[ (f_{an-a\tau_1}(\xi_{a\tau_1}^0 \cup S_{a\tau_1}^0) - f_{an-a\tau_1}(\xi_{a\tau_1}^0)) I(\tau_1 \leq n) \right] \nonumber \\
&\geq& aC_f \P(\tau_1 \leq n). \nonumber 
\end{eqnarray}
Here $I(\tau_1 \leq n)$ is the indicator function of the event $\{\tau_1 \leq n\}$. In the second to last inequality we have used the Markov property and Lemma~\ref{lem: Durrett}, along with the fact that if $\tau_1 \leq n$, then $\xi_{a\tau_1}^0 \cup S_{a\tau_1}^0$ is simply a translate of $A \cup E_a$ where $E_a$ was defined in Lemma~\ref{lem: Durrett}, and $A \subset \widetilde{\mathbb{N}}$ contains a translate of the (almost surely) infinite set $\xi_{a\tau_1}^0$.
\end{proof}
\medskip

We now iterate the construction in the last step to create further $f$-bypasses. Define inductively a sequence of stopping times $(\tau_k)_{k \geq 1} = (\tau_k(S))_{k \geq 1}$ by letting $\tau_{k+1}$ be the smallest $n \geq \tau_k +1$ such that there is the ($C_a$)-configuration around $M_{an-a}^k$. The set $\xi_n^{k+1}$ is defined as 
\[
\xi_n^{k+1} = \begin{cases}
\xi_n^k & \text{ if } n \leq a\tau_{k+1}-1 \\
\xi_{n-a\tau_{k+1}}(\xi_{a\tau_{k+1}}^k \cup S_{a\tau_{k+1}}^k) & \text{ if } n \geq a\tau_{k+1}
\end{cases}\ .
\]
Again, the set  $S_{a\tau_{k+1}}^k=M_{a\tau_{k+1}-a}^k + S^a$ does not intersect $\xi_n^k$. As in the first step, we define $M_n^{k+1}$ to be the point in $\xi_n^{k+1}$ with largest $x$-coordinate and
\[
f_n^{k+1} := f(M_n^{k+1})\ .
\]
Call $K_n= K_n(S)$ the number of $k$'s such that $\tau_k \leq n$. Notice that $K_n$ follows a binomial law with parameters $n$ and $\rho_a$. Then an iteration of the previous step gives us the following lemma.
\begin{lemma}\label{lem:sec} For all $n \in \mathbb{N}$ and any infinite starting set $S$, 
\begin{equation}\label{eq: master}
\E f_{an}^{K_n}(S) - \E f_{an}^0(S) \geq aC_f n \rho_a\ .
\end{equation}
\end{lemma}
\begin{proof}
By definition of $K_n$, we find that $K_n \leq n$. Therefore
\[
\E f_{an}^{K_n} - \E f_{an}^0 = \E f_{an}^n - \E f_{an}^0 = \sum_{k=1}^n \E(f_{an}^k - f_{an}^{k-1})\ .
\] 
Using Lemma~\ref{l:festimate} (the same argument works to give the appropriate bound for arbitrary $k \geq 1$), we get $\E f_{an}^{K_n} - \E f_{an}^0 \geq aC_f\sum_{k=1}^n \P(\tau_k \leq n) = aC_f\E K_n = aC_fn\rho_a$ as desired.
\end{proof}

With all $f$-bypasses constructed and Lemma~\ref{lem:sec} in our hands we can move on to prove Theorem~\ref{thm: diff}.
\medskip

\subsection{Proof of Theorem~\ref{thm: diff}, case $p > \vec p_c$.}
In this subsection we give the proof of Theorem~\ref{thm: diff} in the case $p > \vec p_c$.
We start by estimating the passage time using the path constructed with our bypasses.
For $r \geq 1$, let $b(r)$ be the minimal passage time from the origin to any point in the line
\[
L_r(f) := rL_1(f) = \{(x,y) \in \mathbb{R}^2~:~f(x,y) = r\}.\]

Since $L_1(f)$ is tangent to $B_\mu$ at the point $N_p$, a standard argument using compactness and convexity of the limit shape (see for instance Theorem 6 in \cite{CoxDurrett}) implies that almost surely
\begin{equation}\label{eq: onelimit}
\lim_{n \rightarrow \infty }\frac{b(n)}{n} = \lim_{n \rightarrow \infty }\frac{\tau(0,nN_p)}{n}  = 1.
\end{equation}
For the rest of this subsection, we will write $\xi_n^m$ for $\xi_n^m(\widetilde{\mathbb{N}})$, $M_n^m$ for $M_n^m(\widetilde{\mathbb{N}})$, and so on. In other words, we will make the fixed choice $S= \widetilde{\mathbb{N}}$. Now observe that by construction, whenever the event $\{0 \to D_{an}\}$ occurs, the point $M_{an}^{K_n}$ is connected to the origin by a path with at most $an+2K_n$ edges with passage time equal to $1$. This immediately implies that on the event $\{0 \to D_{an}\}$, we have $b(f_{an}^{K_n})  \leq an+2K_n$.

Since we are assuming that $p > \vec p_c$, $\lim_{n \to \infty} f_{an}^0 = \infty$ almost-surely and therefore we may deduce the following. Almost-surely for $n$ large enough (so that $f_{an}^{K_n}$ is non-zero -- it could vanish for example if $f$ is the function $f(x,y)=x$), on the event $\{0 \to D_{an}\}$,
\begin{equation}\label{eq: endeq1}
\frac{f_{an}^{K_n}}{an} \leq \left( 1+\frac{2}{a} \cdot \frac{K_n}{n} \right) \frac{f_{an}^{K_n}}{b(f_{an}^{K_n})}\ .
\end{equation}

This leads to the following simple observation.
\begin{lemma}\label{lem:lem4} As $n$ goes to infinity,
\[
\E \bigg[I(0 \to D_{an}) \left( 1+\frac{2}{a} \cdot \frac{K_n}{n} \right) \frac{f_{an}^{K_n}}{b(f_{an}^{K_n})}\bigg] \rightarrow (1+(2/a)\rho_a) ~\P(0 \to \infty)\ .
\]
\end{lemma}
\begin{proof}
As $K_n$ is a binomial random variable with parameters $n$ and $\rho_a$, we have almost surely and in $L^1$:
\begin{equation}\label{eq: endeq2}
\frac{K_n}{n} \leq 1 \text{ and } \frac{K_n}{n} \to \rho_a\ .
\end{equation}
Since $\inf supp( \mu) =1$, there exists $C>0$ such that almost-surely (for all $n$ such that $f_{an}^{K_n}$ is non-zero)
\[
\frac{f_{an}^{K_n}}{b(f_{an}^{K_n})} \leq C \ .
\]
(To be completely precise, we can interpret the term on the left side of the above inequality as equal to $C$ when $f_{an}^{K_n}=0$.) Then, by composing limits and using \eqref{eq: onelimit},
\begin{equation}\label{eq: endeq3}
\lim_{n \to \infty} \frac{f_{an}^{K_n}}{b(f_{an}^{K_n})} = 1 \text{ almost surely}\ .
\end{equation}

Because $I(0 \to D_{an})$ converges almost-surely to $I(0 \to \infty)$, \eqref{eq: endeq2}, \eqref{eq: endeq3} and the dominated convergence theorem end the proof of the lemma.
\end{proof}

Next, we will use Lemma~\ref{lem:sec} to get
\begin{proposition} \label{lem:lem5}The following inequality holds:
\begin{equation}
\liminf_{n\rightarrow \infty} \E \bigg[ I(0 \to D_{an}) \frac{f_{an}^{K_n}}{an}  \bigg] \geq \P(0 \to \infty) (1+C_f\rho_a)\ .
\end{equation}
\end{proposition}

However, in order to prove Proposition \ref{lem:lem5} we also need two lemmas about oriented percolation: 

\begin{lemma}\label{lem: bigguys2} Assume that $p > \vec p_c$ and define 
\[
a_n = \sup_{|S|=n, \; S \subseteq \widetilde \N} \Pro (S\nrightarrow \infty)\ .
\]
Then $a_n\to  0$.
\end{lemma}

\begin{remark} A stronger version of this lemma can be found in \cite[Section~10]{Durrett}. We present a proof here for the reader's convenience.
\end{remark}

\begin{proof}
Let $S =  \{v_1, \ldots v_n \}\subseteq \widetilde{\N},$ where we order $S$ by decreasing $x$-coordinate. By translation invariance we may assume that $v_1 =0$. If $k<n/2$, then
\begin{equation*}
\begin{split}
\Pro ( S \rightarrow \infty) &\geq \Pro(0\rightarrow \infty  ) + \Pro( \{ v_{k+1}, \ldots, v_n \} \rightarrow \infty, 0 \nrightarrow D_k )\\ &=  \Pro(0\rightarrow \infty  ) + \Pro( \{ v_{k+1}, \ldots, v_n \} \rightarrow \infty)~  \Pro( 0 \nrightarrow D_k )\ ,
\end{split}
\end{equation*}  
where in the second line we used independence of the two events in question.
Therefore,
\begin{equation*}
\begin{split}
\Pro ( S \nrightarrow \infty) &\leq \Pro(0\nrightarrow \infty  ) - \Pro( \{ v_{k+1}, \ldots, v_n \} \rightarrow \infty) ~\Pro( 0 \nrightarrow D_k )\\ 
&=  \Pro(0\nrightarrow \infty  ) - \Pro( 0 \nrightarrow D_k ) + \Pro( \{ v_{k+1}, \ldots, v_n \} \nrightarrow \infty)~ \Pro( 0 \nrightarrow D_k )\ .
\end{split}
\end{equation*}
Taking the supremum of both sides,
\begin{equation}\label{eq: back}
a_n \leq \Pro(0\nrightarrow \infty) - \Pro(0 \nrightarrow D_k) +  a_{n-k} ~ \Pro(0 \nrightarrow D_k)
\end{equation}
Therefore if we write $A:= \limsup_{n \to \infty} a_n$, we can take $n$ to infinity in \eqref{eq: back} to find
\[
A \leq \P(0 \nrightarrow \infty) - \P(0 \nrightarrow D_k) + A ~\P(0 \nrightarrow D_k)\ .
\]
Letting $k$ tend to infinity, we see that $A \leq A ~\P(0 \nrightarrow \infty)$. Since $p > \vec p_c$, this implies $A=0$ and completes the proof of the lemma.

\end{proof}

\begin{lemma} \label{lem: grim}
Given $\varepsilon >0$, $p>\vec p_c$ and $M >0$, the following inequality holds for all $N$ large enough:
\begin{equation*}
\Pro\bigg( 0 \rightarrow D_N, |\xi_N(\{0\})| < M \bigg) < \varepsilon\ .
\end{equation*} 
\end{lemma}

\begin{proof}
We proceed as in \cite[Lemma~3]{GrimmettMarstrand}. 
\begin{equation*}
\begin{split}
\Pro\bigg( 0 \rightarrow D_N, |\xi_N(\{0\})| < M \bigg) &= \sum_{l=1}^{M-1} \Pro(|\xi_N(\{0\})|=l) = 
\sum_{l=1}^{M-1} \frac{1}{(1-p)^{2l}} \Pro(|\xi_N(\{0\})|=l)(1-p)^{2l}\\ 
&\leq    \frac{1}{(1-p)^{2M}}
\sum_{l=1}^{M-1} \Pro(|\xi_N(\{0\})|=l, \; \text{all edges exiting} \; \xi_N(\{0\}) \; \text{are closed})\\
&=   \frac{1}{(1-p)^{2M}} \sum_{l=1}^{M-1} \Pro(|\xi_N(\{0\})|=l, |\xi_{N+1}(0)|=0) \\
&=   \frac{1}{(1-p)^{2M}} \Pro(|\xi_N(\{0\})|\in [1,M), |\xi_{N+1}(0)|=0)\\
&\leq \frac{1}{(1-p)^{2M}} \Pro(0\rightarrow D_N, 0\nrightarrow \infty) \rightarrow 0, \; \text{as} \; N\rightarrow \infty. 
\end{split}
\end{equation*}

\end{proof}

We now prove Proposition \ref{lem:lem5}.
\begin{proof}[Proof of Proposition  \ref{lem:lem5}]
Given $\varepsilon >0$, we may use Lemma~\ref{lem: bigguys2} to choose $M$ such that whenever $S \subseteq \widetilde \N$ has at least $M$ elements then $\Pro(S \rightarrow \infty) > 1-\varepsilon.$ Next, we use Lemma~\ref{lem: grim} to choose $N$ corresponding to this $M$. Before we proceed, we need to give a bit more notation.

Decompose a configuration $\eta = (\eta_e)$ as $(\eta_{aN},\eta^{aN})$, where $\eta_{aN}$ is the configuration of $\eta$-values for all edges with both endpoints in the set $\{(x,y) \in \mathbb{Z}^2~:~x+y \leq aN\}$ and $\eta^{aN}$ is the configuration of $\eta$-values for all edges with both endpoints in the set $\{(x,y) \in \mathbb{Z}^2:x+y \geq aN\}$. For a configuration $\eta_{aN}$ we make the following two definitions. Let $\sigma = \sigma(\eta_{aN})$ be the set $\xi_{aN}(\{0\})$ in the configuration $\eta$. Also let $\overline \sigma$ be the set $\xi_{aN}^{K_N}(\widetilde{\mathbb{N}})$ in the configuration $\eta$. Notice that $\sigma$ and $\overline \sigma$ depend only on $\eta_{aN}$; they do not depend on $\eta^{aN}$.

For $n > N$, we compute
\begin{eqnarray*}
&&\E \left[ \frac{f_{an}^{K_n}}{an} I(0 \to D_{an}) \right] \geq \E \left[ \E \left[ \frac{f_{an}^{K_n}}{an} I(0 \to D_{an}, |\sigma| > M)~\bigg|~\eta_{aN}\right] \right] \\
&=& \E \left[ \E \left[ \frac{f_{an}^{K_n}}{an} I(0 \to D_{aN}, |\sigma| > M)~I(\sigma \to D_{an})~\bigg|~ \eta_{aN}\right] \right] \\
&=& \E \left[ I(0 \to D_{aN}, |\sigma|>M) ~ \E \left[ \frac{f_{an}^{K_n}}{an} I(\sigma \to D_{an})~\bigg|~\eta_{aN} \right] \right] \\
&=& \E \left[ I(0 \to D_{aN}, |\sigma|>M) ~ \E \left[ \frac{f_{an-aN}^{K_{n-N}(\overline \sigma)}(\overline \sigma)}{an} I(\sigma \to D_{an})~\bigg|~\eta_{aN} \right] \right]\ .
\end{eqnarray*}
We now inspect the inner conditional expectation. For each fixed $\overline \sigma$, $f_{an-aN}^{K_{n-N}(\overline \sigma)}(\overline \sigma)$ is bounded above by $anC_f$ almost surely for all $n$. By our choice of $M$, the following holds for almost every $\eta_{aN}$ on the event $\{|\sigma|>M\}$:
\begin{equation}\label{eq: stapler}
\E\left[f_{an-aN}^{K_{n-N}(\overline \sigma)}(\overline \sigma) I(\sigma \to D_{an}) ~\big|~ \eta_{aN} \right] \geq \E \left[f_{an-aN}^{K_{n-N}(\overline \sigma)}(\overline \sigma) ~\big|~ \eta_{aN}\right] - aC_fn \varepsilon\ .
\end{equation}
By independence of $\eta_{aN}$ and $\eta^{aN}$, the function $f_{an-aN}^{K_{n-N}(\overline \sigma)}(\overline \sigma)$ depends on $\eta_{aN}$ only through $\overline \sigma$. Because $\overline \sigma$ is almost-surely infinite, we may apply Lemma~\ref{lem:sec} to get a lower bound for \eqref{eq: stapler} on the event $\{|\sigma|>M\}$ of
\[
\E \left[ f_{an-aN}^0(\overline \sigma) ~\big|~ \eta_{aN} \right] + aC_f~\E \left[ K_{n-N}(\overline \sigma)~\big|~\eta_{aN} \right] - aC_f n \varepsilon\ .
\]
Since for any fixed $\overline \sigma$ the sequence of points $(M_{an-aN}^0(\overline\sigma)/an)$ converges as $n$ goes to infinity to $N_p$ almost-surely and in $L^1$ (this follows from the definition of $\alpha_p$ from \eqref{eq:NP} -- see \cite{Durrett}), we have $\E \left[\frac{f_{an-aN}^0}{an}(\overline \sigma)~|~ \eta_{aN}\right] \to f(N_p) = 1$ almost surely. In addition, $\E \left[ \frac{K_{n-N}(\overline \sigma)}{n}~|~\eta_{aN} \right]$ converges to $\rho_a$ almost surely. These observations, combined with the above inequalities and Fatou's lemma, imply that
\begin{eqnarray*}
&&\liminf_{n \to \infty} \E \left[ \frac{f_{an}^{K_n}}{an}I(0 \to D_{an}) \right] \\
&\geq& \E \left[ I(0 \to D_{aN},|\sigma|>M) \liminf_{n \to \infty} \E \left[ \frac{f_{an-aN}^{K_{n-N}(\overline \sigma)}(\overline \sigma)}{an} I(\sigma \to D_{an})~\bigg|~\eta_{aN} \right] \right] \\
&\geq& \E \left[ I(0 \to D_{aN},|\sigma|>M) (1+C_f\rho_a - \varepsilon C_f) \right] \\
&=& (1+C_f\rho_a-\varepsilon C_f) \P(|\sigma| >M,~0 \to D_{aN}) \\
&\geq& (1+C_f\rho_a - \varepsilon C_f) \left( \P(0 \to D_{aN}) - \varepsilon) \right) \\
&\geq& (1+C_f\rho_a - \varepsilon C_f) \left( \P(0 \to \infty) - \varepsilon) \right)\ .
\end{eqnarray*}
Taking $\varepsilon$ to zero, the proposition is proved.
\end{proof}

\begin{proof}[Proof of Theorem~\ref{thm: diff}, case $p > \vec p_c$.] Choose 
\begin{equation}\label{defa}
a \geq \max \{ 3/C_f, 2\}.
\end{equation}
Combining Lemma~\ref{lem:lem4} with Proposition~\ref{lem:lem5} and \eqref{eq: endeq1} we get
\[
\P(0 \to \infty) (1+C_f\rho_a) \leq \P(0 \to \infty) (1+(2/a)\rho_a)\ .
\]
Since we are in the case $p > \vec p_c$, $\P(0 \to \infty)>0$, and consequently $3\rho_a \leq 2\rho_a$, which is a contradiction since $\rho_a>0$. Therefore, we are not allowed to choose $f$ as in \eqref{eq: teq} and $\overline L$ is the unique tangent line at the point $N_p$.
\end{proof}

\subsection{Proof of  Theorem~\ref{thm: diff}, case $p=\vec p_c$.}
\begin{proof}[Proof in the case $p=\vec p_c$.] The proof of the case $p = \vec p_c$ is very similar to the previous case, although one should be careful because the oriented cluster of $1$ edges does not percolate. To circumvent this issue we define new variables as in \cite{Marchand}.
 
Let $\e >0$ and $K>1$ be such that
\[
\mu(\tau\leq K) > \vec p_c + \e\ .
\]

Given a realization of passage times $(\tau_e)$ we define $(\eta_e)$ as follows: If $\tau_e = 1$, $\eta_e=1$; if $\tau_e > K$, $\eta_e = 0$; if $\tau_e \in (0,K]$ set $\eta_e$ equal to one with probability $q := \frac{\varepsilon}{\Pro(\tau_e \in (1,K])}$ and $0$ with probability $1-q$.  We make this assignment independently among edges and independently of the random variables $(\tau_e)$.  It follows from this definition that the random variables $(\eta_e)$ are i.i.d. and Bernoulli with parameter $\tilde p~ := \mathbb{P}(\eta_e=1) = \vec p_c + \varepsilon$. Again we embed an oriented percolation model with this parameter $\tilde p$ in our first-passage percolation model. 

As before, we assume that $\partial B_\mu$ is not differentiable at $N_{\vec p_c} = (1/2,1/2)$, so by symmetry of $B_\mu$ about the line $y=x$, we may fix a linear functional $f$ such that (a) $f(N_{\vec p_c}) =1$, (b) $\max_{x \in B_\mu} f(x) = 1$ and (c) $f(1,-1)>0$. 
%
%
%
Fixing $a \geq \max\{3/C_f,2\}$ we can proceed just as in the previous section. On the event $\{0 \to D_{an}\}$ (which means, as before, that $0$ is connected to the line $D_{an}$ using oriented edges with all $\eta$-values equal to 1), let $\gamma_{an}$ be a path that connects  $M_{an}^{K_n}$ to the origin and has at most $an+2K_n$ edges with passage times in $(1,K]$. Note that $\gamma_{an}$ can be defined unambiguously by, for instance, ordering all finite paths in $\mathbb{Z}^2$ deterministically; this will be important because we would like to apply the law of large numbers on this path. Also note that the distributions of $K_n$ and $M_{an}^{K_n}$ depend on $\tilde p$ (that is, they depend on $\e$), which we will take to $\vec p_c$ later. In particular, for fixed $a$, the constant $\rho_a$ depends on $\tilde p$, so we shall write $\rho_a(\tilde p)$, but it has the property that $\lim_{\tilde p \to \vec p_c} \rho_a(\tilde p) > 0$.

Again, for $r\geq 1$, we set $b(r)$ to be the first time any point in the line $L_r(f)$ is reached by first-passage percolation starting at the origin. It follows as before that
\[
\frac{b(n)}{n} \to 1 \text{ almost-surely as } n \to \infty\ .
\]
Equation \eqref{eq: endeq1} becomes almost-surely, for $n$ large enough,
\[
I(0 \to D_{an}) \frac{f_{an}^{K_n}}{an} \leq I(0 \to D_{an}) \left( \frac{1}{an} \sum_{e \in \gamma_{an}} \tau_e \right) \frac{f_{an}^{K_n}}{b(f_{an}^{K_n})}\ .
\]

We can estimate the sum of the weights in $\gamma_{an}$ by noting that 
\[
\frac{1}{an} \sum_{e \in \gamma_{an}} \tau_e  \leq \left(1+\frac{2}{a}\cdot \frac{K_n}{n}\right) \frac{1}{|\gamma_{an}|} \sum_{e \in \gamma_{an}} \tau_e\ .
\]
Therefore we find
\begin{equation}\label{eq: her28}
I(0 \to D_{an}) \frac{f_{an}^{K_n}}{an} \leq I(0 \to D_{an}) \left(1 + \frac{2}{a}\cdot \frac{K_n}{n} \right) \left( \frac{1}{|\gamma_{an}|} \sum_{e \in \gamma_{an}} \tau_e \right) \frac{f_{an}^{K_n}}{b(f_{an}^{K_n})}\ .
\end{equation}

By the law of large numbers for sums of i.i.d. random variables, Lemma~\ref{lem:lem4} becomes 
\begin{lemma}\label{lem:lem6} As $n$ goes to infinity,
\[
\E \bigg[I(0 \to D_{an}) \left(1 + \frac{2}{a}\cdot \frac{K_n}{n} \right) \left( \frac{1}{|\gamma_{an}|} \sum_{e \in \gamma_{an}} \tau_e \right) \frac{f_{an}^{K_n}}{b(f_{an}^{K_n})}\bigg] \rightarrow \P_{\tilde p}(0 \to \infty)~(1+\frac{2}{a}\rho_a(\tilde p))~ \E\left[ \tau_e~|~\eta_e=1\right]\ .
\]
\end{lemma}

On the other hand, the right side in Proposition~\ref{lem:lem5} can be replaced by 
\[
\P_{\tilde p}(0 \to \infty) (1+ C_f\rho_a(\tilde p))\ 
\]
and therefore combining these two, we find 
\begin{equation}\label{e:fff}
\E\left[ \tau_e~|~ \eta_e=1\right] \geq \frac{1+C_f\rho_a(\tilde p)}{1+(2/a)\rho_a(\tilde p)}\ .
\end{equation}
As $\tilde p \to \vec p_c$, the left side approaches $1$. Recall, however, that $a$ was chosen only to depend on $f$ (not on $\tilde p$), so this inequality holds independently of $\tilde p$. Furthermore, $\rho_a(p)$ is a polynomial in $p$. Therefore as $\tilde p \to \vec p_c$,  the right side converges to
\[
\frac{1+C_f\rho_a(\vec p_c)}{1+(2/a)\rho_a(\vec p_c)}\ ,
\]
which is strictly bigger than $1$ since $C_f \geq 3/a$. This gives a contradiction and completes the proof.
\end{proof}

\begin{remark}\label{rem: othercases}
To see that the proof of Theorem \ref{thm: diff} is also valid in the cases (a), (b), (c) of Remark \ref{othercases} one needs to find suitable configurations to play the same role of the previous $f$-bypasses. We will sketch the proof in the case (a), leaving cases (b) and (c) to the reader. We start by (see Figure \ref{fig:Caconfigdir}):

 \begin{definition}
We say that there is the ($\vec C_a$)-configuration around the vertex $(x,y)$ in $\mathbb{Z}^2$ if the following occurs.
\begin{itemize}
\item The edges $\langle(x,y),(x,y+1)\rangle$ and $\langle(x,y),(x+1,y)\rangle$ have $\eta$-value equal to 0, that is, they are closed for  the embedded oriented percolation model.
\item For each vertex $w$ in the set $\{(u,v)~:~u+v=x+y \text{ and } y+1 \leq v \leq  y+a-2\}$, each edge with $w$ as either a left endpoint or a bottom endpoint has $\eta$-value equal to 0. (For $a=2$, this condition is not used.)
\item Each edge between pairs of nearest-neighbor vertices in the set 
\[
\{(u,v) ~:~u \geq x+1, ~v \geq y, \text{ and } \|(u,v)-(x+1,y)\|_1 \leq a-1\}
\]
has $\eta$-value equal to $1$.
\end{itemize}
\end{definition}
\begin{figure}
\centering
\begin{tikzpicture}[line cap=round,line join=round,>=triangle 45,x=0.95cm,y=0.95cm]
\clip(-5.34,-2.1) rectangle (7.13,8.02);
\draw [dotted,domain=-5.34:7.13] plot(\x,{(--2-1*\x)/1});
\draw [dotted,domain=-5.34:7.13] plot(\x,{(--6-2*\x)/2});
\draw [dotted,domain=-5.34:7.13] plot(\x,{(--8-2*\x)/2});
\draw [dotted,domain=-5.34:7.13] plot(\x,{(--5-1*\x)/1});
\draw [dotted,domain=-5.34:7.13] plot(\x,{(--6-1*\x)/1});
\draw [dotted,domain=-5.34:7.13] plot(\x,{(--7-1*\x)/1});
\draw [dash pattern=on 2pt off 2pt] (0,2)-- (1,2);
\draw [dash pattern=on 2pt off 2pt] (1,2)-- (1,1);
\draw [dash pattern=on 2pt off 2pt] (0,2)-- (0,3);
\draw [dash pattern=on 2pt off 2pt] (0,3)-- (-1,3);
\draw [dash pattern=on 2pt off 2pt] (-1,3)-- (-1,4);
\draw [dash pattern=on 2pt off 2pt] (-1,4)-- (-2,4);
\draw [dash pattern=on 2pt off 2pt] (-2,5)-- (-2,4);
\draw [dash pattern=on 2pt off 2pt] (-3,5)-- (-2,5);
\draw (2,2)-- (5,2);
\draw (3,1)-- (3,2);
\draw (2,1)-- (2,2);
\draw (4,1)-- (4,2);
\draw (2,3)-- (2,2);
\draw (2,5)-- (2,3);
\draw (2,3)-- (4,3);
\draw (3,4)-- (3,2);
\draw (4,3)-- (4,2);
\draw (5,1)-- (5,2);
\draw (6,1)-- (2,1);
\draw  [dash pattern=on 2pt off 2pt] (1,1)-- (2,1);
\draw [dash pattern=on 2pt off 2pt] (-3,6)-- (-3,5);
\draw (3,4)-- (2,4);
\draw [color=qqqqcc] (2,0)-- (6,0);
\draw [color=qqqqcc] (-3.84,6)-- (-3.84,2);
\draw [->] (-0.96,-0.25) -- (0.64,0.66);
\draw (-2.01,-0.36) node[anchor=north west] {$\mathit{M_{a\tau_1-a}^0}$};
\draw (4.79,4.33) node[anchor=north west] {$\mathit{S_{a\tau_1}^0}$};
\draw (3.27,-0.1) node[anchor=north west] {$\mathit{a-1}$};
\draw (-5.1,4.15) node[anchor=north west] {$\mathit{a-1}$};
\draw [->,color=qqqqcc] (6.39,1.66) -- (2.94,5.09);
\draw [->,color=qqqqcc] (2.94,5.09) -- (6.39,1.66);
\fill [color=qqqqff] (0,2) circle (1.5pt);
\fill [color=qqqqff] (2,1) circle (1.5pt);
\fill [color=qqqqff] (3,1) circle (1.5pt);
\fill [color=qqqqff] (4,1) circle (1.5pt);
\fill [color=qqqqff] (5,1) circle (1.5pt);
\fill [color=qqqqff] (6,1) circle (1.5pt);
\draw [color=red] (1,1)-- ++(-2.5pt,-2.5pt) -- ++(7.0pt,7.0pt) ++(-7.0pt,0) -- ++(7.0pt,-7.0pt);
\fill [color=qqqqff] (1,2) circle (1.5pt);
\fill [color=qqqqff] (2,2) circle (1.5pt);
\fill [color=qqqqff] (3,2) circle (1.5pt);
\fill [color=qqqqff] (4,2) circle (1.5pt);
\fill [color=qqqqff] (5,2) circle (1.5pt);
\fill [color=qqqqff] (0,3) circle (1.5pt);
\fill [color=qqqqcc] (-1,3) circle (1.5pt);
\fill [color=qqqqcc] (-1,4) circle (1.5pt);
\fill [color=qqqqcc] (-2,4) circle (1.5pt);
\fill [color=qqqqcc] (-2,5) circle (1.5pt);
\fill [color=qqqqcc] (-3,5) circle (1.5pt);
\fill [color=qqqqcc] (2,3) circle (1.5pt);
\fill [color=qqqqcc] (2,5) circle (1.5pt);
\fill [color=qqqqcc] (4,3) circle (1.5pt);
\fill [color=qqqqcc] (3,4) circle (1.5pt);
\fill [color=qqqqcc] (-3,6) circle (1.5pt);
\fill [color=qqqqcc] (2,4) circle (1.5pt);
\fill [color=qqqqcc,shift={(2,0)},rotate=90] (0,0) ++(0 pt,2.25pt) -- ++(1.95pt,-3.375pt)--++(-3.9pt,0 pt) -- ++(1.95pt,3.375pt);
\fill [color=qqqqff,shift={(6,0)},rotate=270] (0,0) ++(0 pt,2.25pt) -- ++(1.95pt,-3.375pt)--++(-3.9pt,0 pt) -- ++(1.95pt,3.375pt);
\fill [color=qqqqff,shift={(-3.84,6)}] (0,0) ++(0 pt,2.25pt) -- ++(1.95pt,-3.375pt)--++(-3.9pt,0 pt) -- ++(1.95pt,3.375pt);
\fill [color=qqqqff,shift={(-3.84,2)},rotate=180] (0,0) ++(0 pt,2.25pt) -- ++(1.95pt,-3.375pt)--++(-3.9pt,0 pt) -- ++(1.95pt,3.375pt);
\fill [color=qqqqcc] (3.02,2.98) circle (1.5pt);
\end{tikzpicture}
\caption{A depiction of the ($\vec C_a$)-configuration around the vertex $M_{a\tau_1-a}^0$ for directed first passage percolation. The dashed bonds have $\eta$-value zero and the full bonds have $\eta$-value one.}
\label{fig:Caconfigdir}
\end{figure}
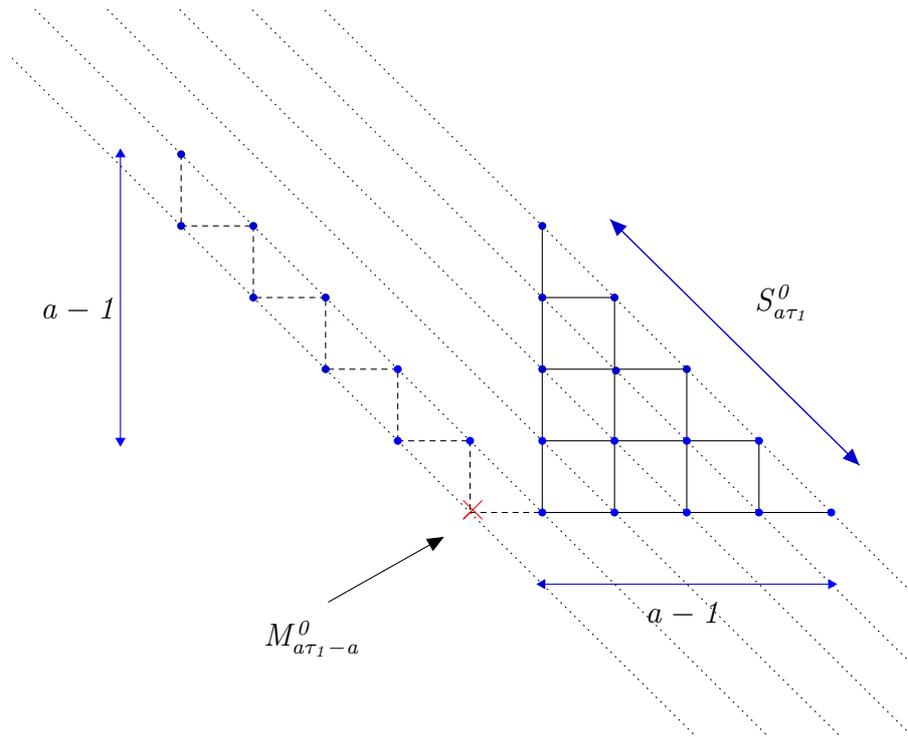

Roughly speaking, a bypass at $(x,y)$ in directed FPP as defined above is a path that passes through a possibly high-weight edge at $\langle (x,y), (x+1,y) \rangle$ so that it can profit from a large configuration of edges with  passage time equal to 1 afterwards.

To complete the proof, choose $a$ such that 
$$ a \geq \max \{2, \frac{3}{C_f}, 2\E (\tau_e | \tau_e >1) \}.$$
The inequality in \eqref{eq: endeq1} is replaced by $$
\frac{f_{an}^{K_n}}{an} \leq \left( a- \frac{K_n}{n}+\frac{1}{a} \cdot \sum_{i=1}^{K_n} \tau_{e_i} \right) \frac{f_{an}^{K_n}}{b(f_{an}^{K_n})}\ , $$
where $e_i$ are the closed edges through which a bypass path pases (they correspond to the edges $\langle(x,y), (x+1,y)\rangle$ in each $(\vec C_a)$ - configuration). By noting that $(\tau_{e_i})_i$ is a collection of independent random variables with mean $\E[\tau_e~|~\tau_e>1]$, one can complete the proof as in the undirected case with minor changes. 
\end{remark}

\section{Proof of Theorem~\ref{thm:unbounded}}\label{sec:Newman}

In this section we use the technique of Newman and Piza to prove Theorem~\ref{thm:unbounded}. Before moving to the proof, we will briefly describe the martingale method used and indicate why their initial work was not able to cover the case treated in this paper. 

Enumerate the edges $(e_i)$ of $\mathbb{Z}^2$ and define a filtration $(\mathcal{F}_i)$, where $\mathcal{F}_i$ is the sigma-algebra generated by the weights $\tau_{e_1}, \ldots, \tau_{e_i}$ and $\mathcal{F}_0$ is the trivial sigma-algebra. Writing $T=\tau(0,nw_\theta)$, we may use $L^2$-orthogonality of martingales to find
\[
\Var T = \sum_{i=1}^\infty \left[\E[T~|~\mathcal{F}_i] - \E[T~|~\mathcal{F}_{i-1}]\right]^2\ .
\]
The $i$-th term in the sum represents a part of the contribution to the variance given by fluctuations of the passage time of the $i$-th edge $e_i$. The idea of \cite{NewmanPiza} is that if $e_i$ is in a geodesic from $0$ to $nw_\theta$ then if we lower its passage time (while keeping all other weights fixed), the variable $T$ will decrease linearly, and therefore $\tau_{e_i}$ will have influence on the fluctuations of $T$. They consequently argue that one has a lower bound for $\Var~T$ of
\[
C~\sum_{i=1}^\infty \P(F_i)^2\ ,
\]
where $F_i$ is the event that $e_i$ is in a geodesic from $0$ to $nw_\theta$. Using clever summation techniques, they bound this below by $C~\log n$.

One main difficulty in extending the Newman-Piza results to the class $\mathcal{M}_p$ is that it is possible that geodesics only use edges with passage times equal to $1$. In this case, we can never lower the weight of an edge in a geodesic. Therefore if we are to use the same technique, we must show that geodesics use many edges with weight above $1$. This is the content of Lemma~\ref{lem:nononeedges}, where such a statement is shown for directions outside the percolation cone. It is not enough only to know this though; if one repeats the computations of \cite{NewmanPiza} using only Lemma~\ref{lem:nononeedges}, one finds only a lower bound of a constant (with no logarithm term). It is essential also to know information about the location of these non-one edges on the lattice. In particular, if they are heavily concentrated enough near the origin, we can extract a logarithmic bound. To do this, we need to know something about the geometry of geodesics (for instance, that they avoid certain regions of the plane). In the proof of Theorem~\ref{thm:unbounded}, we show that the non-one edges of geodesics can be made to lie in certain disjoint annuli of constant aspect ratio. Information of this type turns out to follow from Theorem~\ref{thm: diff} and the non-polygonal nature of $B_\mu$.

\bigskip
\bigskip

We will now state one of the main theorems of \cite{NewmanPiza}. Suppose $T$ is a random variable on a probability space $(\Omega,\mathcal{F},P)$ with $E(T^2)<\infty$, where $E$ represents expectation with respect to $P$. Suppose further that $\Omega = \mathbb{R}^I = \{\omega = (\omega_i~:~i \in I)\}$ with $I$ a countable set and $\mathcal{F}$ is the Borel sigma-field $\mathcal{B}^I$. Generally, for $W \subseteq I$, write $\mathcal{F}(W)$ for the Borel sigma-field $\mathcal{B}^W$. Let $U_1, U_2, \ldots$ be disjoint subsets of $I$ and express $\omega$ for each $k$ as $(\omega^k,\hat\omega^k)$, where $\omega^k$ (resp. $\hat \omega^k$) is the restriction of $\omega$ to $U_k$ (resp. to $I\setminus U_k$). For each $k$, let $D_k^0$ and $D_k^1$ be disjoint events in $\mathcal{B}^{U_k}$. Define
\[
H_k(\omega) = T_k^1(\hat\omega^k) - T_k^0(\hat \omega^k)\ ,
\]
where 
\[
T_k^0(\hat\omega^k) = \sup_{\omega^k \in D_k^0} T((\omega^k,\hat\omega^k)) \text{ and } T_k^1(\hat\omega^k) = \inf_{\omega^k \in D_k^1} T((\omega^k,\hat\omega^k))\ .
\]
\begin{theorem}[Newman-Piza]\label{thm:NP}
Assume the general setting just described and the following three hypotheses about $P$, the $U_k$'s, the $D_k^\delta$'s and $T$:
\begin{enumerate}
\item Conditional on $\mathcal{F}(I \setminus \cup_k U_k)$, the $\mathcal{F}(U_k)$'s are mutually independent.
\item There exist $a,b>0$ such that, for any $k$,
\[
P(\omega^k \in D_k^0~|~\mathcal{F}(U_k^c)) \geq a \text{ and } P(\omega^k \in D_k^1~|~\mathcal{F}(U_k^c)) \geq b \text{ a.s. }
\]
\item For every $k$, $H_k \geq 0$ a.s.
\end{enumerate}
Suppose that, for some $\varepsilon>0$ and each $k$, $F_k \in \mathcal{F}$ is a subset of the event $\{H_k \geq \varepsilon\}$. Then
\[
\Var T \geq ab\varepsilon^2 \sum_k P(F_k)^2\ .
\]
\end{theorem}


Before showing how we will use the last theorem in the proof of Theorem~\ref{thm:unbounded} we will introduce two essential tools given in Proposition~\ref{prop:trapping} and Lemma~\ref{lem:nononeedges}. The first one can be seen as an estimate on the location of first intersection points of growing (random) balls centered at two different points (compare to the proof of Theorem~6 in \cite{NewmanPiza}).



For the rest of this section, fix $\mu \in \mathcal{M}_p$ for some $p \in [\vec p_c, 1)$ and $\theta \in [0,\theta_p)$. For each angle $\phi$, define $v_\phi$ to be the point of $\partial B_\mu$ on the line containing the origin and $w_\phi$. From Corollary~\ref{cor:nonpolygonal}, we can find $\theta_1 \in (\theta,\theta_p)$ such that there is an extreme point of $B_\mu$ at $v_{\theta_0}$ for some $\theta_0 \in (\theta,\theta_1)$. Let $D$ be the closed arc of $\partial B_\mu$ from $\theta_1$ to $-\theta_1$ (counterclockwise) and let $A_n$ be the set of all vertices in geodesics from $0$ to $nv_\theta$. From \eqref{eq:sad}, we may fix $\kappa$ with
\begin{equation}\label{eq: kappaeq}
1/2<\kappa<1
\end{equation}
such that for all $n$ sufficiently large,
\begin{equation}\label{eq: nonrandombound}
\E \tau(0,nv_\theta) - n \leq n^\kappa\ .
\end{equation}
(Though \eqref{eq:sad} is stated for $w_\theta$ we may slightly increase $\kappa<1$ to make the above statement true.)

\begin{figure}\label{fig:pconeD}
\centering
\definecolor{qqwuqq}{rgb}{0,0.39,0}
\definecolor{uququq}{rgb}{0.25,0.25,0.25}
\definecolor{qqqqff}{rgb}{0,0,1}
\begin{tikzpicture}[line cap=round,line join=round,>=triangle 45,x=0.7cm,y=0.7cm]
\clip(-3.42,-3.74) rectangle (12.97,9.55);
\draw [shift={(0,0)},color=qqwuqq,fill=qqwuqq,fill opacity=0.1] (0,0) -- (30.47:0.77) arc (30.47:58.93:0.77) -- cycle;
\draw [line width=0.4pt,dotted] (3.01,4.99)-- (5.04,2.96);
\draw [line width=0.4pt,dotted] (3.01,4.99)-- (0,0);
\draw [line width=0.4pt,dotted] (5.04,2.96)-- (0,0);
\draw [line width=0.4pt,dotted] (5.44,1.96)-- (0,0);
\draw [line width=0.4pt,dotted] (5.44,-1.96)-- (0,0);
\draw[color=qqqqff] (4.01,4.61) node[anchor=north west] {$\omega_{\frac{\pi}{4}}$};
\draw (-3.01,4.99)-- (-5.04,2.96);
\draw [line width=0.4pt,dotted] (-3.01,4.99)-- (0,0);
\draw [line width=0.4pt,dotted] (4.01,4.01)-- (0,0);
\draw [line width=0.4pt,dotted] (-5.04,2.96)-- (0,0);
\draw [line width=0.4pt,dotted] (5.04,-2.96)-- (0,0);
\draw [line width=0.4pt,dotted] (3.01,-4.99)-- (0,0);
\draw [line width=0.4pt,dotted] (3.01,-4.99)-- (5.04,-2.96);
\draw [line width=0.4pt,dotted] (-5.04,-2.96)-- (0,0);
\draw [line width=0.4pt,dotted] (-3.01,-4.99)-- (0,0);

\draw (-3.01,-4.99)-- (-5.04,-2.96);
\draw [shift={(0,0)},dotted]  plot[domain=-0.53:0.53,variable=\t]({1*5.83*cos(\t r)+0*5.83*sin(\t r)},{0*5.83*cos(\t r)+1*5.83*sin(\t r)});
\draw [shift={(0,0)}]  plot[domain=-0.33:0.33,variable=\t]({1*5.83*cos(\t r)+0*5.83*sin(\t r)},{0*5.83*cos(\t r)+1*5.83*sin(\t r)});
\draw [->] (0,0) -- (0.01,8.43);
\draw [->] (-4,0) -- (10.54,0);
\draw (5.93,1.22) node[anchor=north west] {$\bar{D}$};
\fill [color=qqqqff] (3.01,4.99) circle (1.5pt);
\fill [color=qqqqff] (4.01,4.01) circle (1.5pt);
\fill [color=qqqqff] (5.04,2.96) circle (1.5pt);
\draw[color=qqqqff] (5.72,3) node {$\omega_{\theta_p}$};
\draw[color=qqqqff] (6.02,2) node {$\omega_{\theta_1}$};
\draw[color=qqqqff] (6.22,-2) node {$\omega_{-\theta_1}$};
\fill [color=qqqqff] (5.50,1.96) circle (1.5pt);
\fill [color=qqqqff] (5.50,-1.96) circle (1.5pt);
\fill [color=uququq] (0,0) circle (1.5pt);
\fill [color=qqqqff] (-3.01,4.99) circle (1.5pt);
\fill [color=uququq] (0,0) circle (1.5pt);
\fill [color=qqqqff] (-5.04,2.96) circle (1.5pt);
\fill [color=uququq] (0,0) circle (1.5pt);
\fill [color=qqqqff] (5.04,-2.96) circle (1.5pt);
\fill [color=uququq] (0,0) circle (1.5pt);
\fill [color=qqqqff] (3.01,-4.99) circle (1.5pt);
\fill [color=uququq] (0,0) circle (1.5pt);
\fill [color=qqqqff] (3.01,-4.99) circle (1.5pt);
\fill [color=qqqqff] (5.04,-2.96) circle (1.5pt);
\fill [color=qqqqff] (-5.04,-2.96) circle (1.5pt);
\fill [color=uququq] (0,0) circle (1.5pt);
\fill [color=uququq] (0,0) circle (1.5pt);
\fill [color=qqqqff] (-4,0) circle (1.5pt);
\end{tikzpicture}
\caption{Pictorial description of the complement arc  $\bar{D} = \partial B \setminus D$ of $D$. $D$ is the closed arc in $\partial B$ with angles between $\theta_{1}$ and $-\theta_{1}$ counterclockwise.}
\end{figure}
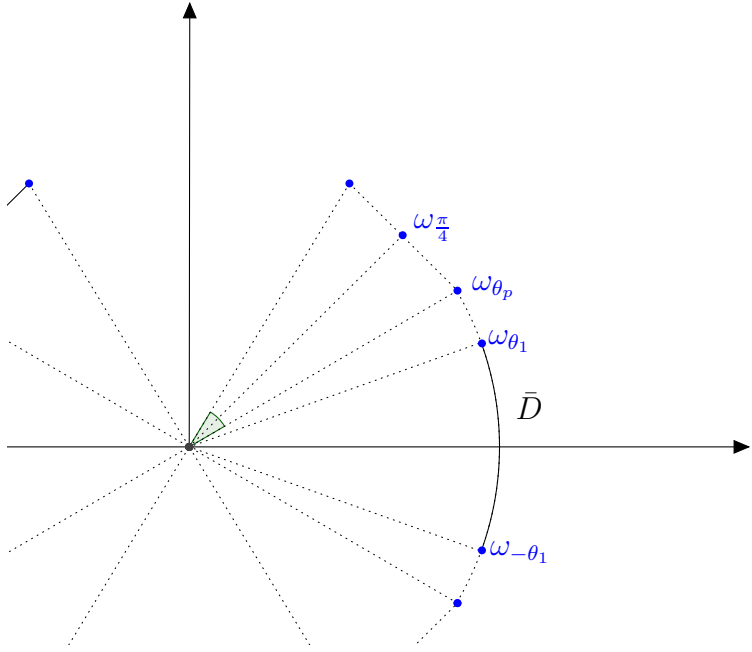


\begin{proposition}\label{prop:trapping}
Given 
\begin{equation}\label{eq:gammadef}
\zeta \in \left( \frac{\kappa+1}{2},1\right)
\end{equation}
and $\e>0$, there exists $n_0=n_0(\zeta,\e)$ such that if $n>n_0$ and $n^\zeta\leq M \leq n/2$ then
\[
\P(A_n \cap M D = \varnothing) > 1-\e\ .
\]
\end{proposition}

\begin{proof}
Since the origin is in the interior of $B_\mu$, we may find a linear functional $f$ that takes its (possibly non-unique) maximum on $B_\mu$ at $v_\theta$ with $f(v_\theta)=1$. We will first argue that $\max_{x \in D} f(x) < 1$. Let $S$ be the set 
\[
S = \{x\in \mathbb{R}^2~:~f(x)=1\}\ .
\]
Assume that $S$ intersects $D$ at some point $v_{\theta_2}$. By convexity of $B_\mu$, all points on $S$ between $v_\theta$ and $v_{\theta_2}$ must be in $B_\mu$. But $B_\mu$ lies entirely on one side of $S$, so these points must be in $\partial B_\mu$. Since $S$ does not intersect the origin, this segment of $S$ is equal to either the clockwise arc or the counterclockwise arc of $\partial B_\mu$ from $\theta$ to $\theta_2$. But then one of the two points $v_{\theta_0}$ or $v_{-\theta_0}$ is contained in the interior of this segment. This is a contradiction since these are extreme points of $B_\mu$. Therefore
\[
\max_{x \in B_\mu} f(x) = f(v_\theta) = 1 \text{ and } a := \max_{x \in D}f(x) <1 \ .
\]

For any $x \in \mathbb{R}^2$, define
\begin{equation}\label{eq:defg}
g(x) = \lim_{n \to \infty} (1/n) \E \tau(0,nx)\ .
\end{equation}
For $x \in \mathbb{Z}^2$, this limit exists by subadditivity of the sequence $(\E \tau(0,nx))_n$ and we even find 
\begin{equation}\label{eq:glatticeeq}
g(x) \leq \E \tau(0,x) \text{ for all } x \in \mathbb{Z}^2\ .
\end{equation}
For non-lattice $x$, we appeal to \cite[Lemma~3.2]{CoxDurrett}, which shows the existence of the limit uniformly over compact subsets of $\mathbb{R}^2$. Strictly speaking, that lemma shows convergence when $\E \tau(0,nx)$ is replaced by a different function (there called $\overline{g}(nx)$) but it satisfies $\sup_{x \in \mathbb{R}^2} | \overline{g}(x) - \E \tau(0,x) | < \infty$, so the result can be applied here.
In fact, $g$ defines a norm on $\mathbb{R}^2$ (see the discussion before Lemma~3.2 of \cite{CoxDurrett} for more properties of $g$) whose closed unit ball is equal to $B_\mu$. To establish a version of \eqref{eq:glatticeeq} for non-lattice $x$, take such an $x$ and write $x'$ for the unique lattice point with $x \in x' + [0,1)^2$. Now
\begin{eqnarray*}
g(x) &=& g(x') + (g(x) - g(x')) \leq \E \tau(0,x') + g(x-x') \\
&\leq& \E \tau(0,x) + \E \tau(0,x'-x) + g(x-x') \leq \E \tau(0,x) + C_1\ ,
\end{eqnarray*}
where $C_1 = \sup_{x,y \in [-2,2]^2} \E \tau(x,y) + \sup_{x \in [-2,2]^2} g(x)$. Similar reasoning using translation invariance of $\P$ by lattice vectors shows that for all $x,y \in \mathbb{R}^2, |\mathbb{E}\tau(x,y) - \mathbb{E}\tau(0,y-x)|\leq C_1$. Therefore
\begin{equation}\label{eq:translateerror}
\left|\mathbb{E} \tau(x,y) - \mathbb{E} \tau(x+z,y+z)\right| \leq 2C_1 \text{ for all } x,y,z \in \mathbb{R}^2\ .
\end{equation}
Finally this gives
\begin{equation}\label{eq:gnonlatticeeq}
g(x) \leq \E \tau(y,y+x) + 3C_1 \text{ for all } x,y \in \mathbb{R}^2\ .
\end{equation}

If $v \in MD$, we estimate 
\[
\E \tau(v,nv_\theta) \geq g(nv_\theta-v) - 3C_1\ .
\]
But for all $x\in \R^2$, $g(x) \geq f(x)$. This is clearly true if $g(x)=0$. Otherwise, $x/g(x) \in B_\mu$ so $f(x/g(x)) \leq 1$. Therefore 
\begin{eqnarray}\label{eq:expecteq1}
\E \tau(v,nv_\theta) &\geq& f(nv_\theta-v) -3C_1 = n-f(v) -3C_1 \nonumber \\
&\geq& n-aM -3C_1 = (n-M) +(1-a)M - 3C_1\ .
\end{eqnarray}
Now by \eqref{eq:translateerror} and \eqref{eq: nonrandombound}, there exists $C_2>0$ such that for all $0\leq M\leq n$,
\begin{equation}\label{eq:expecteq2}
\E \tau(Mv_\theta,nv_\theta) \leq g((n-M)v_\theta) + C_2(n-M)^\kappa + 2C_1 \leq (n-M) + C_2n^\kappa + 2C_1\ .
\end{equation}
Therefore if $\frac{2C_2n^\kappa}{1-a} + \frac{10C_1}{1-a}\leq M \leq n$, we combine \eqref{eq:expecteq1} and \eqref{eq:expecteq2} to find
\begin{equation}\label{eq:expecteq3}
\E \tau(Mv_\theta,nv_\theta) \leq \E \tau(v,nv_\theta) - (1-a)M/2 \text{ for all } v \in MD\ .
\end{equation}
Last, since $(1/M)\mathbb{E} \tau(0,Mv_{\theta'}) \to 1$ uniformly in $\theta'$ (again this follows from \cite[Lemma~3.2]{CoxDurrett}), we may choose $n_1$ such that $n \geq n_1$ and $n^\zeta \leq M \leq n/2$ implies
\begin{equation}\label{eq:expecteq4}
\E \tau(0,Mv_\theta) + \E \tau(Mv_\theta,nv_\theta) \leq \E \tau(0,v) + \E \tau(v,nv_\theta) - (1-a)M/4 \text{ for all } v \in MD\ .
\end{equation}
Here we are using the fact that $\zeta > \kappa$.


The following lemma will be used to bound the values of $|\tau(0,v)-\mathbb{E}\tau(0,v)|$ for $v \in MD$. Its proof will be given in the appendix. Fix $\zeta'<\zeta$ that still satisfies \eqref{eq:gammadef}.
\begin{lemma}\label{lem: ballbound}
Given $\e>0$, there exists $M_0$ such that for all $M \geq M_0$,
\[
\mathbb{P}\left( |\tau(0,x)-\mathbb{E}\tau(0,x)| \leq M^{\zeta'} \text{ for all } x \in M\partial B_\mu \right) > 1-\e\ .
\]
\end{lemma}


For any $n,M$, let $B_{n,M}$ be the event that
\begin{enumerate}
\item[(A)] $\left| \tau(0,v)-\mathbb{E} \tau(0,v) \right| \leq n^{\zeta'}$ for all $v \in MD$ and
\item[(B)] $\left| \tau(v,nv_\theta) - \mathbb{E} \tau(v,nv_\theta) \right| \leq n^{\zeta'}$ for all $v \in MD$.
\end{enumerate}
By Lemma~\ref{lem: ballbound} we may choose $n_0\geq n_1$ such that $n \geq n_0$ and $n^\zeta \leq M < n/2$ implies $\mathbb{P}(B_{n,M}) >1-\e$. 
We may further increase $n_0$ if necessary so that if $n \geq n_0$ then $(1-a)n^{\zeta}/4 \geq 2n^{\zeta'}$. We claim that for such $n$ and $M$, $B_{n,M}$ is the desired event for Proposition~\ref{prop:trapping}. Indeed, suppose that $v \in MD \cap A_n$. Then
\[
\tau(0,v) + \tau(v,nv_\theta) = \tau(0,nv_\theta) \leq \tau(0,Mv_\theta) + \tau(0,nv_\theta)\ .
\]
Using items (A) and (B) above, we find 
\[
\E \tau(0,v) + \E \tau(v,nv_\theta) \leq \E \tau(0,Mv_\theta) + \E \tau(Mv_\theta,nv_\theta) + 2 n^{\gamma'}\ .
\]
Combining this with \eqref{eq:expecteq4}, we find $(1-a)M/4 \leq 2 n^{\zeta'}$, which is a contradiction since $M\geq n^\zeta$.
\end{proof}

Our next tool is a lemma that gives a lower bound on the number of edges with values strictly larger than $1$ through which a geodesic outside the percolation cone must pass. Before stating it we must introduce an ordering on measures in $\mathcal{M}$, following the notation of  \cite{Marchand} and \cite{vdbk}.
\begin{definition}
For $\mu$ and $\nu$ in $\mathcal{M}$ we say that $\mu$ is {\it more variable} than $\nu$, written $\nu \prec \mu$, if
\[
\int f~d\mu \leq \int f~d\nu
\]
for all concave increasing functions $f~:~\mathbb{R} \to \mathbb{R}$ for which both sides are defined. If $\mu \neq \nu$ then we say that $\mu$ is {\it strictly more variable} than $\nu$.
\end{definition}
This ordering on measures was used by van den Berg and Kesten in \cite{vdbk} to show relations between limit shapes. Specifically, it was shown that if $\mu$ and $\nu$ satisfy certain mild requirements then if $\mu$ is strictly more variable than $\nu$, we have the strict containment $B_\nu \subset int~ B_\mu$. One of the requirements is that neither measure is in $\mathcal{M}_p$. Later, Marchand extended these results to the class $\mathcal{M}_p$. To describe this, recall that $\theta_p$ is the unique angle in $[0,\pi/4]$ such that $N_p$ is on the line which contains the origin and the point $w_{\theta_p}=(\cos \theta_p, \sin \theta_p)$. For any $x \in \mathbb{R}^2$ we previously defined (see \eqref{eq:defg})
\[
g(x) = \lim_{n \to \infty} \frac{1}{n} \E \tau(0,nx)\ .
\]
To emphasize dependence of $g$ on $\mu$ we will write $g_\mu$. The following is a restatement of the comparison theorem of Marchand. For symmetry reasons it is again only stated for the first quadrant.
\begin{theorem}[Marchand]\label{thm:marchand2}
Let $\mu$ and $\nu$ be in $\mathcal{M}_p$ such that $\mu$ is strictly more variable than $\nu$. Then for all $\theta \in [0,\theta_p)$, 
\[
g_\nu(v_\theta) > g_\mu(v_\theta)\ .
\]
\end{theorem}

\begin{remark}
In fact a slightly stronger statement follows from Theorem~\ref{thm:marchand2} and properties of the limit shape. Let $C$ be a closed subset of $\partial B_\mu$ that does not intersect the flat edge $[M_p,N_p]$ (or any of its reflections in the axes). Let $\Theta$ be the set of angles corresponding to $C$ (that is, $\theta \in \Theta$ if the line segment connecting $0$ to a point in $C$ makes the angle $\theta$ with the $x$-axis). If $\mu$ and $\nu$ are in $\mathcal{M}_p$ such that $\mu$ is strictly more variable than $\nu$, then
\[
\inf_{\theta \in \Theta} \left[ g_\nu(v_\theta) - g_\mu(v_\theta) \right] > 0\ .
\]
\end{remark}

Now we are ready to state our next tool, a lemma that is analogous to Claim~3 in \cite{DH}. 
Since $p<1$ we may choose $y>1$ such that 
\begin{equation}\label{eq:defy}
q:=\mu([y,\infty))>0\ .
\end{equation}

\begin{lemma}\label{lem:nononeedges}
Let $C$ be a closed subset of $\partial B_\mu$ that does not intersect $[M_p,N_p]$ or any of its reflections about the axes. Given $\varepsilon>0$, there exists $M_1$ and $\rho>0$ such that with probability at least $1-\varepsilon$, the following holds. For all $M>M_1$, $x \in MC$ and for every geodesic $\gamma$ from $0$ to $x$, at least $\rho M$ edges of $\gamma$ have passage times $\geq y$.
\end{lemma}

\begin{proof}
Define the edge weights $\{\tau_e':e \in \mathbb{E}^2\}$ by the rule that if $\tau_e \geq y$ then $\tau_e' = \tau_e+1$ and $\tau_e'=\tau_e$ otherwise. Let $\mu'$ be the marginal distribution of $\tau_e'$. It is not difficult to see that $\mu$ is strictly more variable than $\mu'$. Therefore, by Theorem~\ref{thm:marchand2} (and the remark afterward), we may find $\eta>0$ such that $(1-\eta) g_{\mu'}(v_\theta) > g_\mu(v_\theta)$ for all $\theta \in \Theta$, the set of angles corresponding to the set $C$. Equivalently, we have 
\[
(1-\eta)C \cap B_{\mu'} = \varnothing\ .
\]

For a path $\gamma$ let $N(\gamma)$ be the number of edges in $\gamma$ with weight at least $y$. By the shape theorem, there exists an event $A$ with $\mathbb{P}(A) > 1-\varepsilon$ and $M_1$ such that for all $M>M_1$ and $x \in M \partial B_\mu$, the $\tau$-geodesic $\gamma$ from $0$ to $x$ satisfies $(1-\eta^2)M < \tau(\gamma) < (1+\eta^2)M$, and similarly for $x' \in M\partial B_{\mu'}$ and the $\tau'$-length of $\tau'$-geodesics from $0$ to $x'$. We claim that $A$ is the desired event. Indeed, let $M>M_1$ and let $\gamma$ be a $\tau$-geodesic from $0$ to some $z \in MC$. We have
\[
\tau'(\gamma) = \tau(\gamma) + N(\gamma) \leq (1+\eta^2)M + N(\gamma)\ .
\]
On the other hand, $z=\frac{M}{s}x$ for some $x \in \partial B_{\mu'}$ and $s<1-\eta$, so
\[
\tau'(\gamma) \geq (1-\eta^2)\frac{M}{1-\eta}\ .
\]
Combining these we find that $N(\gamma)\geq (\eta-\eta^2)M$. We take this to be $\rho M$.
\end{proof}

We are now in good shape to give the proof of Theorem~\ref{thm:unbounded}.

\begin{proof}[Proof of Theorem~\ref{thm:unbounded}]

Let $e_1, e_2, \ldots$ be an enumeration of the edges of $\mathbb{E}^2$. We will first apply Theorem~\ref{thm:NP}. For that purpose, we set $T = \tau(0,n_\theta)$, $U_k = \{e_k\}$, $D_k^0$ the event that $\tau_{e_k} = 1$ and $D_k^1$ the event that $\tau_{e_k} \geq y$. Last, set $\varepsilon = y-1$, $a=p$ and $b=q$. The reader may verify that all of the hypotheses of Theorem~\ref{thm:NP} are satisfied. Therefore, setting
\[
F_k = \{\tau_{e_k} \geq y \text{ and } e_k \text{ is in a geodesic from } 0 \text{ to } nv_\theta\}\ ,
\]
we find that
\begin{equation}\label{eq:Fk}
\Var \tau(0,nv_\theta) \geq pq (1-y)^2 \sum_k \P(F_k)^2\ .
\end{equation}

Next, fix $\zeta$ as in Proposition~\ref{prop:trapping} and write $C$ for the arc of $\partial B_\mu$ complementary to $D$ (that is, the arc of $B_\mu$ from $\theta_1$ to $-\theta_1$ clockwise). Given $\e>0$, Proposition~\ref{prop:trapping} gives an $n_0$ such that for all $n>n_0$ and $n^\zeta \leq M \leq n/2$, the probability is at least $1-\e$ that all geodesics from $0$ to $nv_\theta$ intersect $M\partial B_\mu$ only in $MC$. Given any such geodesic, the piece of it from 0 to the first intersection with $M\partial B_\mu$ is a geodesic contained in $MB_\mu$, and so as long as $M \geq M_1$ (from Lemma~\ref{lem:nononeedges}), with probability at least $1-\e$ it will contain at least $\rho M$ edges with passage times at least $y$.

Summarizing, given $\e>0$ we may choose $n_0$ such that $n>n_0$ and $n^\zeta\leq M \leq n/2$ implies that with probability at least $1-\e/2$,
\begin{enumerate}
\item[(A')] every geodesic from $0$ to $nv_\theta$ contains at least $\rho M$ edges with endpoints in $MB_\mu$ with weights at least equal to $y$.
\end{enumerate}
Furthermore, from the shape theorem and the fact that $\inf~supp(\mu)=1$ (and therefore the number of edges in a path is at most its passage time), we may choose $c>1$ and $L_0>0$ such that with probability at least $1-\e/2$,
\begin{enumerate}
\item[(B')] for all $L>L_0$ every geodesic from $0$ to $nv_\theta$ contains at most $cL$ edges in $LB_\mu$.
\end{enumerate}

We now define a sequence of numbers $(\hat M_i)$ by $\hat M_1 = n^\zeta$ and, setting $J=2c/\rho$, for $i=1, \ldots, \mathcal{I}$ (for $\mathcal{I}=\lfloor (1-\zeta)/\log J\rfloor \log (n/2)$),
\[
\hat M_{i+1} = J \hat M_i\ .
\]

From \eqref{eq:Fk} it follows that if we write $\mathcal{E}_i$, $i=0, \ldots, \mathcal{I}-1$ for the set of edges with an endpoint in $\hat M_{i+1} B_\mu$ but no endpoints in $\hat M_i B_\mu$ (for $i=0$ we take all edges with an endpoint in $\hat M_i B_\mu$), then
\[
\Var \tau(0,nv_\theta) \geq pq(1-y)^2 \sum_{i=0}^{\mathcal{I}-1} \sum_{k:e_k \in \mathcal{E}_i} \P(F_k)^2\ .
\]
Using Jensen's inequality, we get a lower bound of
\begin{equation}\label{eq:sum1}
pq(1-y)^2 \sum_{i=0}^{\mathcal{I}-1} \frac{1}{|\mathcal{E}_i|}\left(\sum_{k:e_k \in \mathcal{E}_i} \P(F_k)\right)^2\ .
\end{equation}
We will now give a lower bound for the inner sum. Call $X_{i+1}$ the event that (A') and (B') hold for $M=\hat M_{i+1}$. On this event the number of edges on any geodesic $\gamma$ from $0$ to $nv_\theta$ in the set $\hat M_iB_\mu$ is at most $(\rho/2) \hat M_{i+1}$ and the number of edges on $\gamma$ in $\hat M_{i+1}B_\mu$ with weight at least $y$ is at least $\rho \hat M_{i+1}$. From this it follows that on the event $X_{i+1}$, we have the lower bound
\[
\sum_{k:e_k \in \mathcal{E}_i} I(F_k) \geq (\rho/2) \hat M_{i+1}\ .
\]
Since $B_\mu$ is bounded, there exists $d>0$ such that for all $i$, $\hat M_i \geq d \sqrt{|\mathcal{E}_i|}$. So \eqref{eq:sum1} is bounded below by
\[
pq(1-y)^2\left[ d(\rho/2) (1-\e)\right]^2\mathcal{I} \geq \widetilde C \log n
\]
for $\widetilde C$ independent of $n$. This completes the proof.

\end{proof}

\section{Proof of Theorem~\ref{thm:powerlaw}}

Again we apply Theorem~\ref{thm:NP} using the same setup as last time, so that \eqref{eq:Fk} holds. Let $\xi'>\xi_\theta$. Write $\mathcal{D}_n$ for the set of edges with an endpoint in $\Lambda_n^\gamma$ and neither of whose endpoints has distance more than $2n g_\mu(v_\theta)$ from the origin. Let $G_n$ be the event that $M_n(\theta)$ (recall the definition from after \eqref{eq:chitheta}) is contained in the set of endpoints of edges in $\mathcal{D}_n$. By the shape theorem and \eqref{eq:xitheta} we may find $C_1>0$ such that for all $n$ sufficiently large, $\mathbb{P}(G_n)>C_1$.

Choose $\varepsilon<C_1$ and let $\rho$ be from Lemma~\ref{lem:nononeedges}. We now essentially repeat the proof of \cite[Theorem~5]{NewmanPiza}, using Cauchy-Schwarz on \eqref{eq:Fk}. It is bounded below by
\begin{eqnarray*}
pq(1-y)^2 \sum_{k:e_k \in \mathcal{D}_n} \mathbb{P}(F_k)^2 &\geq& pq \frac{(1-y)^2}{|\mathcal{D}_n|}\left( \sum_{k:e_k \in \mathcal{D}_n} \mathbb{P}(F_k) \right)^2 \\
&=& pq \frac{(1-y)^2}{|\mathcal{D}_n|} \left[ \mathbb{E}\left( \sum_{k:e_k \in \mathcal{D}_n} I(F_k) \right)\right]^2 .
\end{eqnarray*}
Let $B_n$ be the event from Lemma~\ref{lem:nononeedges} using $C$ equal to the intersection of the ray at angle $\theta$ with $B_\mu$. We get now a lower bound of
\[
pq\frac{(1-y)^2}{|\mathcal{D}_n|} (\rho n \mathbb{P}(G_n \cap B_n))^2 \geq C_2 n^{1-\xi'}
\]
for $n$ sufficiently large. This completes the proof.

\appendix

\section{Appendix: proof of Lemma~\ref{lem: ballbound}}

Recall that $\mathbb{E}\tau_e^2 < \infty$ and that there exists $\kappa$ with $1/2<\kappa<1$ and $C_1>0$ such that for all $x \in \mathbb{R}^2$,
\begin{equation}\label{eq: assumption2}
|g(x) - \mathbb{E} \tau(0,x)| \leq C_1n^{\kappa}\ .
\end{equation}

For each $M>0$ let $\widetilde B(M)$ be the set of lattice points within distance 2 of $M\partial B_\mu$. We will show that there exists $M_0$ such that for all $M \geq M_0$,
\begin{equation}\label{eq: propeq}
\mathbb{P}\left( |\tau(0,x)-\mathbb{E}\tau(0,x)| \leq M^{\zeta'} \text{ for all } x \in \widetilde B(M) \right)>1-\e
\end{equation}
and this will prove Lemma~\ref{lem: ballbound}.

By \cite[Theorem~1]{Kesten}, there exists $C_2>0$ such that for all $a,x \in \mathbb{R}^2$, $Var~\tau(a,a+x) \leq C_2 |x|$. Therefore by Chebyshev, for any $\lambda>0$ and $a,x \in \mathbb{R}^2$,
\begin{equation}\label{eq: tool1}
\mathbb{P}(|\tau(a,a+x)-\mathbb{E}\tau(a,a+x)| \geq \lambda) \leq C_2\frac{|x|}{\lambda^2}\ .
\end{equation}

For $x \in \mathbb{R}$ and $r>0$ write 
\[
B(x;r) = \{z \in \mathbb{R}^2:|x-z|\leq r\}\ .
\]
Last, choose $\kappa''>\kappa'>\kappa$ such that 
\[
\zeta' > \frac{\kappa''+1}{2}\ .
\]

We now give a geometric lemma. The proof will be postponed until the end of the appendix.
\begin{lemma}\label{lem: convex}
There exists $C_3$ with the following two properties. 
\begin{enumerate}
\item For each $M>0$ and $y \in \widetilde B(M)$ the number of points in $\widetilde B(M) \cap B(y;M^{\kappa'})$ is at most $C_3M^{\kappa'}$.
\item For each $M>0$ one may choose $N_M:=\lfloor C_3M^{1-\kappa'} \rfloor $ (deterministic) points $x_1, \ldots, x_{N_M} \in \widetilde B(M)$ such that
\[
\widetilde B(M) \subseteq \bigcup_{k=1}^{N_M} B(x_k;M^{\kappa'})\ .
\] 
\end{enumerate}
\end{lemma}

For $M>0$, choose $C_3$ and $x_1, \ldots, x_{N_M}$ as in Lemma~\ref{lem: convex}. For $k=1, \ldots, N_M$ let $B_{k,M}$ be the event
\[
B_{k,M} = \left\{ \text{for all } y\in B(x_k;M^{\kappa'}) \cap \widetilde B(M), ~\tau(x_k,y) \leq M^{(1+\kappa'')/2}\right\} \ .
\]

Now we estimate
\begin{eqnarray*}
\mathbb{P}(B_{k,M}^c) &\leq& \sum_{y \in B(x_k,M^{\kappa'}) \cap \widetilde B(M)} \mathbb{P}\left(\tau(x_k,y) \geq M^{(1+\kappa'')/2}\right) \\
&\leq& \sum_{y \in B(x_k,M^{\kappa'}) \cap \widetilde B(M)} \mathbb{P}\left( |\tau(x_k,y) - \mathbb{E}\tau(x_k,y)| \geq M^{(1+\kappa'')/2}-\mathbb{E}\tau(x_k,y)\right)
\end{eqnarray*}
We can choose $M_1$ such that if $M \geq M_1$ then for all $k=1, \ldots, N_M$ and all $y \in B(x_k,M^{\kappa'})$, we have
\[
M^{(1+\kappa'')/2} - \mathbb{E}\tau(x_k,y) \geq (1/2) M^{(1+\kappa'')/2}\ .
\]
This follows from the fact that $\kappa'<\kappa''$ and that there exists $C_4$ such that for all $x,y \in \mathbb{R}^2$, $\mathbb{E}\tau(x,y) \leq C_4|x-y|$. Therefore for $M \geq M_1$, we can use part 1 of Lemma~\ref{lem: convex} and \eqref{eq: tool1}:
\begin{eqnarray*}
\mathbb{P}(B_{k,M}^c) &\leq& \sum_{y \in B(x_k,M^{\kappa'}) \cap \widetilde B(M)} \mathbb{P}\left( |\tau(x_k,y) - \mathbb{E}\tau(x_k,y)| \geq (1/2) M^{(1+\kappa'')/2}\right) \\
&\leq& 4C_2 \sum_{y \in B(x_k,M^{\kappa'}) \cap \widetilde B(M)} \frac{M^{\kappa'}}{M^{1+\kappa''}} \leq \frac{4C_2C_3}{M^{1+\kappa'' - 2\kappa'}}\ .
\end{eqnarray*}
From this estimate and the value of $N_M$ from Lemma~\ref{lem: convex}, we find for $M \geq M_1$,
\[
\mathbb{P}(B_{k,M} \text{ for all } k = 1, \ldots, N_M) > 1-\frac{4C_2C_3^2}{M^{\kappa''-\kappa'}}\ .
\]
Therefore there exists $M_2 \geq M_1$ such that for all $M \geq M_2$,
\begin{equation}\label{eq: tool2}
\mathbb{P}(B_{k,M} \text{ for all } k = 1, \ldots, N_M) > 1-\e/2\ .
\end{equation}

Define the events $A_{k,M}$ for $M>0$ and $k=1, \ldots, N_M$ as
\[
A_{k,M} = \left\{|\tau(0,x_k)-\mathbb{E}\tau(0,x_k)| \leq M^{(1+\kappa'')/2}\right\}\ .
\]
Using \eqref{eq: tool1} again for each $k=1, \ldots, N_M$ and for any $M>0$, $\mathbb{P}(A_{k,M}^c) \leq \frac{C_2}{M^{\kappa''}}$. From the value of $N_M$ from Lemma~\ref{lem: convex},
\[
\mathbb{P}(A_{k,M} \text{ for all } k = 1, \ldots, N_M) > 1- \frac{C_3M}{M^{\kappa'+\kappa''}}
\]
and so using $\kappa''>\kappa'>1/2$, we can choose $M_3 \geq M_2$ such that for all $M \geq M_3$,
\begin{equation}\label{eq: tool3}
\mathbb{P}(A_{k,M} \text{ for all } k = 1, \ldots, N_M) > 1-\e/2\ .
\end{equation}

We now claim that there exists $M_0 \geq M_3$ such that if $M \geq M_0$ and the events in \eqref{eq: tool2} and \eqref{eq: tool3} hold (which occurs with probability at least $1-\e$) then the event in \eqref{eq: propeq} holds. Once we show this, the proof will be complete.

Let $M>0$ and suppose that the events of \eqref{eq: tool2} and \eqref{eq: tool3} hold. If $x \in \widetilde B(M)$ then we can use part 2 of Lemma~\ref{lem: convex} to find $k$ between $1$ and $N_M$ such that $x \in B(x_k;M^{\kappa'})$. Because $B_{k,M}$ occurs, $\tau(x_k,x) \leq M^{(1+\kappa'')/2}$. Therefore
\[
|\tau(0,x) - \tau(0,x_k)| \leq \tau(x_k,x) \leq M^{(1+\kappa'')/2}\ .
\]
Using the triangle inequality,
\begin{eqnarray*}
|\tau(0,x) - \mathbb{E}\tau(0,x)| &\leq& M^{(1+\kappa'')/2} + |\tau(0,x_k) - \mathbb{E}\tau(0,x_k)| \\
&+& |\mathbb{E}\tau(0,x_k) - g(x_k)| + |\mathbb{E}\tau(0,x)-g(x)| \\
&+& |g(x_k) - g(x)|\ .
\end{eqnarray*}
By the definition of $\kappa$ and the fact that $A_{k,M}$ occurs, there exists $C_5$ such that for all $M$ and all $x \in \widetilde B(M)$, we have an upper bound of
\begin{equation}\label{eq: last1}
2M^{(1+\kappa'')/2} + C_5 M^\kappa + |g(x_k) - g(x)|\ .
\end{equation}
Because $x_k$ and $x$ are within distance $2$ of $MB_\mu$,
\[
|g(x_k) - M| \leq C_6 \text{ and } |g(x) - M| \leq C_6\ ,
\]
where $C_6 = \sup_{y \in [-2,2]^2} g(y)$. So $|g(x_k) - g(x)| \leq 2C_6$. Putting this in \eqref{eq: last1} and using the fact that $\kappa < (1+\kappa'')/2 < \zeta'$, we may find $M_0 \geq M_3$ such that for all $M \geq M_0$ and all $x \in \widetilde B(M)$, the expression is bounded above by $M^{\zeta'}$. This completes the proof.

\begin{proof}[Proof of Lemma~\ref{lem: convex}]

The first statement follows from the fact that the arc length of the boundary of a convex set contained in a ball of radius $r$ cannot exceed $2\pi r$ (see for instance \cite[p.15]{Jaglom}). 
%
For the second statement, note first that  the number of vertices of $\widetilde B(M)$ is at most $C_4M$ for some $C_4$. Now we argue similarly to Howard and Newman \cite{HowardNewman}: place balls $B_i = B(x_i;M^{\kappa'})$ for points $x_i$ in $\widetilde B(M)$ in the following manner. Choose $x_1 \in \widetilde B(M)$ to be arbitrary and, assuming we have chosen $x_1, \ldots, x_k$ then pick $x_{k+1} \in \widetilde B(M)$ so that
\[
B(x_{k+1};M^{\kappa'}/2) \subseteq \left( \cup_{i=1}^k B(x_i;M^{\kappa'}/2) \right)^c\ .
\]
Because at each step, the ball $B(x_k;M^{\kappa'}/2)$ covers at least $C_5M^{\kappa'}$ number of vertices in $\widetilde B(M)$ and these balls are disjoint, the procedure must end after at most $N_M : = \lfloor C_6M^{1-\kappa'} \rfloor$ number of iterations. But now 
\[
\widetilde B(M) \subseteq \cup_{k=1}^{N_M} B(x_k;M^{\kappa'})\ .
\]
\end{proof}

\bigskip
\noindent
{\bf Acknowledgements:} We thank Jack Hanson and Tom Lagatta for discussions and for careful readings of the paper. We also thank an anonymous referee for a detailed report that contributed to a better presentation of this manuscript.


\begin{thebibliography}{1}

\bibitem{Alexander} Alexander, K. (1997). Approximation of subadditive functions and convergence rates in limiting-shape results. {\it Ann. Probab.} {\bf 1} 30-55.


\bibitem{Blair-Stahn} Blair-Stahn, N. D. (2010). First passage percolation and competition models. {\it arXiv:1005.0649}.

\bibitem{CD} Chatterjee, S. and Dey, P. (2009). Central limit theorem for first-passage percolation time across thin cylinders. {\it arXiv:0911.5702}.

\bibitem{CoxDurrett} Cox, J. T. and Durrett, R. (1981). Some limit theorems for percolation with necessary and sufficient conditions. {\it Ann. Probab.} {\bf 9} 583-603.

\bibitem{DH} Damron, M. and Hochman, M. (2010) Examples of non-polygonal limit shapes in i.i.d. first-passage percolation and infinite coexistence in spatial growth models. {\it arXiv:1009.2523}.

\bibitem{Durrett} Durrett, R. (1984). Oriented percolation in two dimensions. {\it Ann. Probab.} {\bf 12} 999-1040.

\bibitem{DurrettLiggett} Durrett, R. and Liggett, T. (1981). The shape of the limit set in Richardson's growth model. {\it Ann. Probab.} {\bf 9} 186-193.

\bibitem{Eden} Eden, M. A two-dimensional growth process. 1961 {\it Proc. 4th Berkeley Sympos. Math. Statist. and Prob., Vol. IV} pp. 223--239 {\it Univ. California Press, Berkeley, Calif.}

\bibitem{GM} Garet, O. and Marchand, R. (2005). Coexistence in two-type first-passage percolation models. {\it Ann. Appl. Probab.}, {\bf 15}(1A) 298-330.

\bibitem{Gouere} Gou\'er\'e, J.-B. (2007). Shape of territories in some competing growth models. {\it Ann. Appl. Probab.}, {\bf 17}(4) 1273-1305.

\bibitem{Grimmett} Grimmett, G. (1999). {\it Percolation}. 2nd edition. {\it Springer, Berlin.}

\bibitem{GrimmettMarstrand} Grimmett, G. and Marstrand, J. M. (1990). The supercritical phase of percolation is well behaved. {\it Proc. Roy. Soc. London Ser. A} {\bf 430} 439--457.

\bibitem{HM} H\"aggstr\"om, O.; Meester, R.  (1995). Asymptotic shapes for stationary first passage percolation. {\it Ann. Probab.} {\bf 23} 1511--1522. 

\bibitem{HP} H\"aggstr\"om, O. and Pemantle, R. (1998). First passage percolation and a model for competing spatial growth. {\it J. Appl. Probab.} {\bf 35} 683-692.

\bibitem{HW}  Hammersley, J. and  Welsh, D. First-passage percolation, subadditive processes, stochastic networks, and generalized renewal theory. 1965 {\it Proc. Internat. Res. Semin., Statist. Lab., Univ. California, Berkeley, Calif.} pp. 61--110 {\it Springer-Verlag, New York}.


\bibitem{Hoffman} Hoffman, C. (2008). Geodesics in first passage percolation. {\it Ann. Appl. Probab.} {\bf 18} 1944-1969.

\bibitem{Hoffman1} Hoffman, C. (2005). Coexistence for Richardson type competing spatial growth models. {\it Ann. Appl. Probab.}, {\bf 15}(1B) 739-747.

\bibitem{Howard} Howard, C. Douglas. Models of first-passage percolation. {\it Probability on discrete structures,} 125--173, Encyclopaedia Math. Sci., 110, {\it Springer, Berlin,} 2004.

\bibitem{HowardNewman} Howard, C. Douglas and Newman, Charles M. (1997). Euclidean models of first-passage percolation. {\it Probab. Theory Relat. Fields} {\bf 108} 153--170.

\bibitem{Johan} Johansson, K.  (2000). Shape fluctuations and random matrices. {\it Comm. Math. Phys.}, {\bf 209} 437--476.

 \bibitem{KPZ}  Kardar, K. and Parisi, G. and Zhang, Y. (1986). Dynamic scaling of growing interfaces.  {\em Phys. Rev. Lett.}  {\bf 56} 889--892.



\bibitem{Kesten} Kesten, H. (1993). On the speed of convergence in first-passage percolation. {\it Ann. Appl. Probab.} {\bf 3} 296-338.

\bibitem{Kestensurvey} Kesten, H. First-passage percolation. {\it From classical to modern probability,} 93--143, Progr. Probab., 54, {\it Birkh\"auser, Basel,} 2003.

\bibitem{Marchand} Marchand, R. (2002). Strict inequalities for the time constant in first passage percolation. {\it Ann. Appl. Probab.} {\bf 12} 1001-1038.

\bibitem{Newman} Newman, C. A surface view of first-passage percolation. Proceedings of the International Congress of Mathematicians, Vol. 1, 2 (Z\"urich, 1994), 1017-1023, Birkh\"auser, Basel, 1995.

\bibitem{NewmanPiza} Newman, C. and Piza, M. (1995). Divergence of shape fluctuations in two dimensions. {\it Ann. Probab.} {\bf 23} 977-1005.


\bibitem{PP} Pemantle, R. and Peres, Y. (1994). Planar first-passage percolation times are not tight. {\it Probability and phase transition (Cambridge, 1993)}, 261-264, NATO Adv. Sci. Inst. Ser. C Math. Phys. Sci., 420, {\it Kluwer Acad. Publ., Dordrecht}.


\bibitem{Pimentel} Pimentel, L. P. R. (2007). Multitype shape theorems for first passage percolation models. {\it Adv. in Appl. Probab.}, {\bf 39}(1) 53-76.

\bibitem{Richardson} Richardson, D. (1973). Random growth in a tessellation. {\it Proc. Cambridge Philos. Soc.} {\bf 74} 515-528.

\bibitem{Jaglom} Yaglom, I. and Boltjanski V. (1960). {\it Convex figures}, Holt, Rinehart and Winston, New York.

\bibitem{vdbk} Van den Berg, J. and Kesten, H. (1993). Inequalities for the time constant in first-passage percolation. {\it Ann. Appl. Probab.} {\bf 3} 56-80.

\bibitem{Zhang} Zhang, Yu. (2008). Shape fluctuations are different in different directions. {\it Ann. Probab.} {\bf 36} 331-362.

\bibitem{Zhang2} Zhang, Yu. (2010). On the concentration and the convergence rate with a moment condition in first-passage percolation. {\it Stoch. Proc. Appl.} {\bf 120} 1317-1341.

\bibitem{Zhang3} Zhang, Yu. (2009). The time constant vanishes only on the percolation cone in directed first passage percolation. {\it Electron. J. Prob.}{\bf 77} 2264-2286.

\end{thebibliography}
\end{document}